\documentclass[11pt]{amsart}
\usepackage[babel]{csquotes}
\usepackage{enumitem}
\usepackage{geometry}
\usepackage{graphicx} % Required for inserting images
\usepackage{amsmath}
\usepackage{amssymb}
\usepackage{amsthm} 
\usepackage{hyperref}
\usepackage{indentfirst}

\numberwithin{equation}{section}

\newtheorem{mainthm}{Theorem}
\newtheorem*{theorem*}{Theorem}
\newtheorem{theorem}{Theorem}[section]
\newtheorem{lemma}[theorem]{Lemma}
\newtheorem{corollary}[theorem]{Corollary}
\newtheorem{prop}[theorem]{Proposition}

\theoremstyle{definition}
\newtheorem{defn}[theorem]{Definition}

\theoremstyle{remark}
\newtheorem{rem}[theorem]{Remark}

\newtheorem*{ques}{Question}

\title{Surface Quotients of Right-Angled Hyperbolic Buildings}

\author{Donghae Lee}
\address{Donghae ~Lee. Department of Mathematical Sciences, Seoul National University, {\it eastsea6574@snu.ac.kr}}

\begin{document}
\maketitle

\begin{abstract}
In this paper, we develop the theory of surface quotients of Fuchsian buildings, a hyperbolic building in which each chamber is a regular $p$-gon. Our focus is on the right-angled Fuchsian building $I_{p, \mathbf{q}}$, defined by a $p$-tuple $\mathbf{q}$ of integers, which indicates the thickness of each side of the chamber. We proved some sufficient conditions for the existence of a surface quotient for a given compact hyperbolic surface tessellated by right-angled regular $p$-gons. By combining these results with tessellation methods of hyperbolic surfaces, We find some conditions on $\mathbf{q}$ under which a surface quotient of $I_{p, \mathbf{q}}$ exists. We also define certain symmetry conditions of $\mathbf{q}$ and prove that these symmetries are necessary for the existence of a surface quotient of $I_{p, \mathbf{q}}$ when the number of $p$-gons to tessellate the surface quotient is not divisible by 4.
\end{abstract}

\section{Introduction}\label{sec1}
A 2-dimensional hyperbolic building is a negatively curved 2-dimensional cell complex constructed from hyperbolic right-angled polygons \cite{futer2012surface}. In this paper, we will discuss the right-angled Fuchsian building $I_{p, \mathbf{q}}$: it is the unique simply-connected 2-dimensional complex such that
\begin{itemize}
    \item every 2-cell is a right angled regular $p$-gon with edges indexed by $i=1, \cdots, p$.
    \item the link of a vertex at the intersection of type $i$ edge and type $i+1$ edge is the complete bipartite graph $K_{q_i, q_{i+1}}$ for given $\mathbf{q}=\{q_1, q_2, \cdots, q_p \}$. (Here, $q_{p+1}:=q_{1}$.)
\end{itemize}
See Definition \ref{def_Ipq} for the precise definition.

Fuchsian buildings are spaces with many interesting properties, and their automorphism groups exhibit a rich variety of symmetries. Among these buildings, a specific type known as the very special right-angled Fuchsian building—which has a constant branching number $v$ along every edge of its polygons and is called Bourdon's building—has been the focus of extensive research. Bourdon's building $I_{p, v}$, which can be regarded as a right-angled Fuchsian building $I_{p, {v, v, \dots, v}}$, represents a negatively curved space with numerous intriguing characteristics. There has been prior interest in the lattices of Bourdon's buildings even before the work of Futer and Thomas, like Bourdon-Pajot \cite{bourdon1999poincare}, Haglund \cite{haglund2002existence}, \cite{haglund2006commensurability}. 

Consider a discrete subgroup $\Gamma$ of $G$ (with respect to the the compact-open topology).
There is significant interest in lattices (i.e. discrete subgroups of finite covolume) of the automorphism group $G$ of a hyperbolic building, which has been studied by Thomas \cite{thomas508385uniform}, Abramenko-Peter-Brown \cite{abramenko2009automorphisms}, Caprace-Emmanuel-Monod \cite{caprace2009isometry}, Carbone-Kangaslampi-Vdovina \cite{carbone2012groups}, Ciobotaru \cite{ciobotaru2014flat}, Kangaslampi-Vdovina \cite {kangaslampi2017hyperbolic}, Smith\cite{smith2018rigidity}, Norledge-Thomas-Vdovina \cite{norledge2018maximal}, Savela \cite{savela2020finding}, etc. 

When the quotient space $I_{p, \mathbf q}/\Gamma$ is a compact surface without boundary, such a lattice $\Gamma$ is referred to as a \textit{surface quotient lattice} and will be denoted by $\Gamma_{p,\mathbf{q}}$. Each cell of the quotient space $I_{p, \mathbf q}$ by $\Gamma$ has a finite stabilizer if $\Gamma$ acts cocompactly. (See Section \ref{subsec2.2} below.) If the genus of the quotient space is $g$, the lattice $\Gamma$ is called a genus $g$ surface quotient lattice. The study of surface quotient lattices of hyperbolic buildings began with Futer and Thomas \cite{futer2012surface} and has since been further developed by Constantine-Lafont-Oppenheim \cite{constantine2014surface}, Kangaslampi \cite{kangaslampi2014surface}, and Kangaslampi-Vdovina \cite{kangaslampi2017hyperbolic}. 

Our main theorem extends the results of Futer and Thomas \cite{futer2012surface}, which concern some sufficient conditions for the existence of genus $g$ surface quotient lattices for Bourdon's buildings $I_{p, v}$. The following theorem, which is the main result in \cite{futer2012surface}, describes the condition on $p$, $v$, and $g$ for the existence of a genus $g$ lattice $\Gamma_{p, v, g}$ of $Aut(I_{p, v})$.

\begin{theorem*}[{Futer-Thomas\cite{futer2012surface}}]\label{thm_thomas}
    Let $p\geq5$, $v\geq2$, and $g\geq2$ be integers, and let $I_{p, v}$ be Bourdon's building with a constant branching number $v$. Assume that $F=\displaystyle\frac{8(g-1)}{p-4}$ is a positive integer.
    \begin{enumerate}
        \item Existence of $\Gamma_{p, v, g}$.
        \begin{enumerate}
            \item If $v\geq2$ is even, then for all $F$, a lattice $\Gamma_{p, v, g}$ exists.
            \item\label{refine} If $F$ is divisible by 4, then for all integers $v\geq 2$, a lattice $\Gamma_{p, v, g}$ exists.
            \item If $F$ is composite, then for infinitely many odd integers $v\geq 3$, a lattice $\Gamma_{p, v, g}$ exists.
        \end{enumerate}
        \item Non-existence of $\Gamma_{p, v, g}$.
        \begin{enumerate}
            \item If $F$ is odd, then for infinitely many odd integers $v\geq3$, a lattice $\Gamma_{p, v, g}$ does not exist.
        \end{enumerate}
    \end{enumerate}
\end{theorem*}

In this paper, we provide some sufficient conditions for the existence of a genus $g$ surface quotient lattice $\Gamma_{p, \mathbf{q}, g}$ in $Aut(I_{p, \mathbf{q}})$ for general $\mathbf{q}$. Additionally, we refine case \ref{refine} of Futer and Thomas's result by restricting it to the setting of Bourdon's building. Main Theorem \ref{mainthm_exist} provides sufficient conditions for the existence of $\Gamma_{p,\mathbf{q}, g}$. Before stating the main theorems, let us define some sufficient conditions on $\mathbf{q}$ for the existence of a lattice $\Gamma_{p, \mathbf{q}, g}$.

\begin{defn}[alternating non-coprime sequence]\label{def_altseq}
    A sequence $\mathbf{q}=\{ q_1, \cdots , q_p \}$ is called an alternating non-coprime sequence if it belongs to one of the following two cases:
    \begin{enumerate}
        \item\label{def_altseq_1} If $p$ is even, $q_1 , q_3, \cdots , q_{p-1}$ are all divisible by $\exists d\in \mathbb{N}_{\geq 2}$, and $q_2 , q_4 , \cdots , q_p$ are all divisible by an integer $e>1$. 
        \item\label{def_altseq_2} If $p$ is odd, $q_{i+1} , q_{i+3}, \cdots , q_{i+p-2}$ are all divisible by $\exists d\in \mathbb{N}_{\geq 2}$, and $q_{i+2} , q_{i+4}, \cdots , q_{i+p-1}$ are all divisible by an integer $e>1$ for some integer $i$ between 1 and $p$. Here, indices are $\mathrm{mod}\; p$.
    \end{enumerate}
\end{defn}

\begin{defn}[2-symmetric or 4-symmetric sequence]\label{def_symmseq}
    Let $p$ be an even integer and $\mathbf{q}=\{q_1, q_2, \cdots, q_p\}$ be an integer sequence with length $p$. If the sequence $\mathbf{q}=\{q_{1}, q_{2}, \cdots , q_{p}\}$ satisfies $q_{m+i}=q_{m-i}$ for some integer $m$ and all $i$, then $\mathbf{q}$ is called \emph{2-symmetric} \emph{(about $m$)}. If $p$ is a multiple of 4 and an appropriate positive integer $m$ exists such that $q_{m+i}=q_{m+{p/2}-i}=q_{m+{p/2}+i}=q_{m-i}$ is satisfied, ${q_{i}}$ is called a \emph{4-symmetric} \emph{(about $m$)}. (Here, the indices of $q_i$'s are up to $\mathrm{mod}\;p$.)
\end{defn}

Our first main theorem provides a sufficient condition for the existence of $\Gamma_{p, \mathbf{q}, g}$ in terms of alternating non-coprime sequences and symmetric sequences.

\begin{mainthm}[sufficient conditions for existence of $\Gamma_{p, \mathbf{q}, g}$]\label{mainthm_exist}
    Let $I_{p, \mathbf{q}}$ be a right-angled Fuchsian building determined by a $p$-tuple of integers $\mathbf{q}=\{q_1, q_2, \cdots , q_p\}$, where $p$ is an even integer. Denote by $\Gamma_{p, \mathbf{q}, g}$ a genus $g$ surface quotient lattice of $Aut(I_{p, \mathbf{q}})$, if it exists, for some integer $g$ such that $F=\frac{8(g-1)}{p-4}$ is an integer. Then, the following statements hold:
    \begin{enumerate}
        \item\label{mainthm_exist_1} If $F$ is divisible by 4, for all alternating non-coprime $\mathbf{q}$, a lattice $\Gamma_{p, \mathbf{q}, g}$ exists.
        \item\label{mainthm_exist_2} If $F$ is even and not divisible by 4, then for all alternating non-coprime and 2-symmetric $\mathbf{q}$, a lattice $\Gamma_{p, \mathbf{q}, g}$ exists.
        \item\label{mainthm_exist_3} If $F$ is odd composite number, then for all alternating non-coprime and 4-symmetric alternating non-coprime $\mathbf{q}$, a lattice $\Gamma_{p, \mathbf{q}, g}$ exists if two alternating gcd $d$ and $e$ of $\mathbf{q}$ are both even.
    \end{enumerate}
\end{mainthm}

Note that $F$ is the number of faces of the quotient space. (See Lemma \ref{lem_facenum}.) The following corollaries from Theorem \ref{mainthm_exist} are immediate.

\begin{corollary}[Every Fuchsian building generated by an even-length alternating non-coprime sequence covers a surface]\label{mainthm_cor1}
    For any even integer $p \geq 6$ and any alternating non-coprime sequence $\mathbf{q}=\{q_1, q_2, \cdots , q_p\}$, the automorphism group $\emph{Aut}(I_{p, \mathbf{q}})$ admits a lattice whose quotient is a compact, orientable hyperbolic surface.
\end{corollary}

\begin{corollary}[Every genus space is covered by a Fuchsian building]\label{mainthm_cor2}
    For any $g \geq 2$ and even $p \geq 6$ such that $\frac{8(g-1)}{p-4}$ is an integer, there exists a compact, orientable hyperbolic surface of genus $g$ that is the quotient of a Fuchsian building with $p$-gons as chambers.
\end{corollary}

Also, we obtain an improved version of statement \ref{refine} from Futer-Thomas's existence theorem by restricting to the case $\mathbf{q} = \{v, v, \cdots, v\}$. This result follows directly from the fact that a tuple consisting solely of the integer $v \geq 2$, $\{v, v, \cdots, v\}$, is always an alternating non-coprime, 2-symmetric, and 4-symmetric sequence when $p$ is divisible by 4.

\begin{corollary}[Refinement of \ref{refine} from the main theorem in \cite{futer2012surface}]
    Let $p \geq 5$, $v \geq 2$, and $g \geq 2$ be integers such that $F = \displaystyle\frac{8(g-1)}{p-4}$ is a positive integer. If $F$ is even, then a lattice $\Gamma_{p, v, g}$ exists for all integers $v \geq 2$.
\end{corollary}

Theorem \ref{mainthm_nonexist} provides some necessary conditions on $p$, $\mathbf{q}$, and $g$ for existence of $\Gamma_{p, \mathbf{q}, g}$.

\begin{mainthm}[necessary condition for existence of $\Gamma_{p, \mathbf{q}, g}$]\label{mainthm_nonexist}
Let $I_{p, \mathbf{q}}$ be a right-angled Fuchsian building determined by a $p$-tuple of integers $\mathbf{q}=\{q_1, q_2, \cdots , q_p\}$. Denote by $\Gamma_{p, \mathbf{q}, g}$ a genus $g$ surface quotient lattice of $\emph{Aut}(I_{p, \mathbf{q}})$, if it exists, for some integer $g$ such that $F=\frac{8(g-1)}{p-4}$ is an integer. Then, the following statements hold:
    \begin{enumerate}
        \item If $F$ is even and not divisible by 4, then $\mathbf{q}$ must be 2-symmetric for a lattice $\Gamma_{p, \mathbf{q}, g}$ to exist.
        \item If $F$ is odd, then $\mathbf{q}$ must be 4-symmetric for a lattice $\Gamma_{p, \mathbf{q}, g}$ to exist.
    \end{enumerate}
\end{mainthm}

Note that Main Theorem \ref{mainthm_exist} and Main Theorem \ref{mainthm_nonexist} are not equivalent. A natural question that arises is whether conditions in Main Theorem \ref{mainthm_exist} are also necessary conditions for the existence of $\Gamma_{p, \mathbf{q}, g}$ or not. More briefly, the following question can be proposed:
\begin{ques}
    Is an alternating non-coprime condition of $\mathbf{q}$ necessary for the existence of $\Gamma_{p, \mathbf{q}, g}$ when $p$ is even?
\end{ques}

The results in this paper provide a new perspective on the group action on hyperbolic buildings. In particularly, if there is a symmetry in the sequence of thickness $\mathbf{q}$, referred to as the \emph{Alternating non-coprime condition}(See Definition \ref{def_altseq}), the surface quotient lattice $\Gamma_{p, \mathbf{q}, g}$ exists. Additionally, our proof differs from the previous work of Futer-Thomas \cite{futer2012surface} in that it employs an inductive method, decomposing elements of the homology group into smaller loops, referred to as \emph{polygonal loops} and \emph{boundary geodesic loops} (See Definition \ref{def_bbl} and Definition \ref{def_bgl}).
\\

This article is organized as follows: In Section \ref{sec2}, we recall preliminary facts on complexes of groups. Section \ref{subsec2.2} covers the method of Futer and Thomas in \cite{futer2012surface} for tessellating hyperbolic surfaces with right-angled regular $p$-gons. In Section \ref{sec3}, we prove a homological property about the closed geodesic loops in the 1-skeleton of tessellation of a compact hyperbolic surface, which will be used in Section \ref{sec4}. In Section \ref{sec4}, we provide a general theorem (Theorem \ref{prop_mainthm}) on the existence of genus $g$ lattices, applying the tessellation method described in Section \ref{subsec2.2}, from which Main Theorem \ref{mainthm_exist} follows. Finally, in Section \ref{sec5}, we prove necessary conditions on $p$, $\mathbf{q}$, and $g$ for the existence of a lattice $\Gamma_{p, \mathbf{q}, g}$.

\section{Preliminaries}\label{sec2}

In this section, we recall the theory of complexes of groups and the definition of building structure from \cite{abramenko2008buildings} section 4. We also define the right-angled Fuchsian building, which is the main object of this paper, following the definition in \cite{bourdon2000immeubles} Section 1.5.1. Then, in \ref{subsec2.2}, we gather known results about hyperbolic tessellations of closed (i.e. compact without boundary) surfaces.

\subsection{Complexes of Groups}\label{subsec2.0}
Before defining a right-angled Fuchsian building, which is a main concept in our paper, we first briefly summarize the theory of complexes of groups. Complexes of groups can be viewed as a high-dimensional Bass-Serre theory of graph of groups (\cite{serre2002trees}). However, a key difference exists between Bass-Serre theory and complexes of groups theory, known as \textit{developability}.

\begin{defn}[Complexes of Groups, \cite{futer2012surface}]
    A complex of groups $G(X)=(G_{\sigma}, \psi_{a}, g_{a, b})$ over a polygonal complex $X$ is given by:
    \begin{enumerate}
        \item a group $G_{\sigma}$ for each $\sigma\in V(X')$ where $X'$ is a barycentric subdivision of $X$, called the \emph{local group} at $\sigma$;
        \item a monomorphism $\psi_{a} : G_{i(a)}\rightarrow G_{t(a)}$ for each $a \in E(X')$; and
        \item for each pair of composable edges $a, b$ in $X'$, an element $g_{a, b}\in G_{t(a)}$, such that 
        \begin{center}
            $\text{Ad} (g_{a, b})\circ \psi_{ab}=\psi_{a}\circ\psi_{b}$
        \end{center}
        where $\text{Ad} (g_{a, b})$ is conjugation by $g_{a ,b}$ in $G_{t(a)}$.
    \end{enumerate}
\end{defn}
Let $X$ be a locally finite polyhedral complex. Suppose that there is a simply connected polyhedral complex $\Tilde{X}$ and $\Gamma \le \text{Aut}(\Tilde{X})$ such that $X=\Tilde{X}/\Gamma$. Let $G(X)$ denote the complex of groups generated by adding the data about stabilize; for each polyhedral subcomplex, a stabilizer subgroup of $\Gamma$-action on $\Tilde{X}$ is assigned. For any polyhedral complex $Y$ and complexes of groups $G(Y)$, We call $G(Y)$ \textit{developable} if there is a simply connected polygonal complex $\Tilde{X}$ that makes $G(Y)\simeq G(X)$ for some $\Gamma$. If the dimension of $X$ is 1, $X$ can be considered as a graph, and any $G(X)$ is developable. However, for 2 or higher dimensional $X$, $G(X)$ does not need to be developable. It is known that every complex of groups with nonpositive curvature is locally developable. (See \cite{martinez2013coherence} for precise definition.) The important point to note here is we don't need to consider the global developability of $G(X)$ in this paper, because we are only focusing on the hyperbolic case.

However, even when $G(X)$ is not developable, we can define \textit{local development} at each vertex $\sigma$ on $X$. Before defining local development, let us denote by $V(X)$ the vertex set of $X$ and $X'$ the first Barycentric subdivision of $X$. For $\sigma \in V(X)$, \textit{star} $\text{St}(\sigma)$ of $\sigma$ is defined as the union of all simplices in $X'$ that meet with $\sigma$. 
\begin{defn}[Local Development]\label{def_local development}
    Let $\widetilde{\text{St}(\sigma)}$ be a complex of groups with an action of $\Gamma_{\sigma}$ that is free of edge inversions, such that the quotient $\widetilde{\text{St}(\sigma)}/\Gamma_{\sigma}$ is isomorphic to $\text{St}(\sigma)$. We define $\widetilde{\text{St}(\sigma)}$ as the \emph{local development} of $\sigma$ in $X$, and denote it by $\text{St}(\Tilde{\sigma})$.
\end{defn}
Proof for the existence of local development can be found in \cite{ziemianska2015coverings}.

For a given 2-dimensional complex $X$ and vertex $\sigma \in X$, the \textit{link} of $\sigma$, denoted by $Lk(\sigma, X)$, is defined as the graph obtained by intersecting $X$ with a sufficiently small sphere $B_r (\sigma)$ centered at $\sigma$. Link tells us how $X$ looks like near $\sigma$, so it reflects the local characteristics of $X$. For example, in the building structure that we will define below (see Definition \ref{def_building}), the link of any vertex is a connected graph.

From the local developability of complex of groups mentioned above, it follows that it is possible to define the link of the local development for any $\sigma\in X$. The link of the local development will later be considered significantly to determine the given complexes of groups' universal cover, if it exists.

However, a universal cover always exists if the given complex of groups has nonpositive curvature: This is one of the important properties of negatively curved spaces.

\begin{theorem}[Bridson-Haefliger \cite{bridson2013metric}]\label{thm_nonpositive developable}
    A nonpositively curved complexes of groups is developable.
\end{theorem}

\subsection{Right-angled Fuchsian building}\label{subsec2.1}

Now we will provide the general definition of 2-dimensional building. This definition of building is followed by a specification of this definition for right-angled Fuchsian buildings, the main concept of this paper, using a certain complex of groups and its universal cover.

\begin{defn}[building, \cite{abramenko2008buildings}]\label{def_building}
    A \emph{building} is a simplicial complex $\Delta$ that can be expressed as the union of subcomplexes $\Sigma$ (called \emph{apartments}) satisfying the following axioms:
    \begin{enumerate}
        \item Each apartment $\Sigma$ is a Coxeter complexes.
        \item For any two simplices $A, B \in \Delta$ (called \emph{chambers}), there is an apartment $\Sigma$ containing both of them. 
        \item If $\Sigma$ and $\Sigma'$ are two apartments containing $A$ and $B$, then there is an isomorphism $\Sigma\rightarrow\Sigma'$ fixing $A$ and $B$ pointwise.
    \end{enumerate}
\end{defn}

\begin{rem}
    In the case of a 2-dimensional building $\Delta$, it can be defined as a polygonal complex consisting of polygonal chambers. The barycentric subdivision of $\Delta$ then coincides with the definition of a building as described in Definition \ref{def_building}. Throughout this paper, only 2-dimensional buildings are considered, so the term \emph{building} refers to the definition based on a polygonal complex.
\end{rem}

Definition \ref{def_building} provides a general definition for arbitrary-dimensional building structures. By adding the condition that each chamber in the building $\Delta$ is a 2-dimensional polygon, this definition becomes that of a 2-dimensional building. Buildings can also be categorized by the curvature of their apartments as well as the dimension of the chambers. If each apartment has positive (zero, negative, respectively) curvature, the building is called a spherical (Euclidean, hyperbolic) building. This paper focuses on 2-dimensional hyperbolic buildings, which are of significant interest in the fields of geometric and combinatorial group theory. Specifically, we discuss \emph{right-angled Fuchsian buildings}, studied in \cite{farb2008problems}, \cite{haglund2006commensurability}, \cite{bourdon2000rigidity}, \cite{remy2004kac}, \cite{remy2006topological}, \cite{capdeboscq2012cocompact}, \cite{remy2012buildings}, \cite{lim2006enumeration}, \cite{clais2014right}, \cite{clais2016combinatorial}. A right-angled Fuchsian building is a 2-dimensional hyperbolic building where each chamber is a regular right-angled $p$-gon, which is briefly defined in Section \ref{sec1}. Definition \ref{def_Ipq} is the definition of right-angled Fuchsian buildings from Section 1.5.1 in \cite{bourdon2000immeubles}.

\begin{defn}[right-angled Fuchsian building $I_{p, \mathbf{q}}$, \cite{bourdon2000immeubles}]\label{def_Ipq}
    Let $p$ be an integer with $p\geq 5$, let $P$ be a right-angled regular $p$-gon in hyperbolic plane $\mathbb{H}^2$ with curvature -1, and let $\mathbf{q}=\{ q_1, q_2, \cdots , q_p \}$ be a $p$-tuple of integers so that $q_i \geq 3$ for $i=1, 2, \cdots, r$. Then label clockwise the edges of $P$ by $\{ 1\}, \{2\}, \cdots, \{p\}$ and its vertices by $\{1, 2\}, \{2, 3\}, \cdots, \{p-1, p\}, \{p, 1\}$, where a vertex which is placed on the intersection of an edge $\{i\}$ and an edge $\{i+1\}$ is labeled as $\{i, i+1\}$. 
    Assign the trivial group to the face of $P$, the group $\Gamma_i=\mathbb{Z}/q_i\mathbb{Z}$ to edge $\{i\}$, and the group $\Gamma_{i, i+1}=\Gamma_i\times \Gamma_{i+1}$ to vertex $\{i, i+1\}$. The resulting complex of groups, obtained by assigning groups in this manner, is denoted as $G(P)$. The universal cover of $G(P)$ is called a \emph{right-angled Fuchsian building}, denoted by $I_{p, \mathbf{q}}$.
\end{defn}

\begin{figure}\label{fig: indexing}
    \centering
    \includegraphics[width=0.7\textwidth]{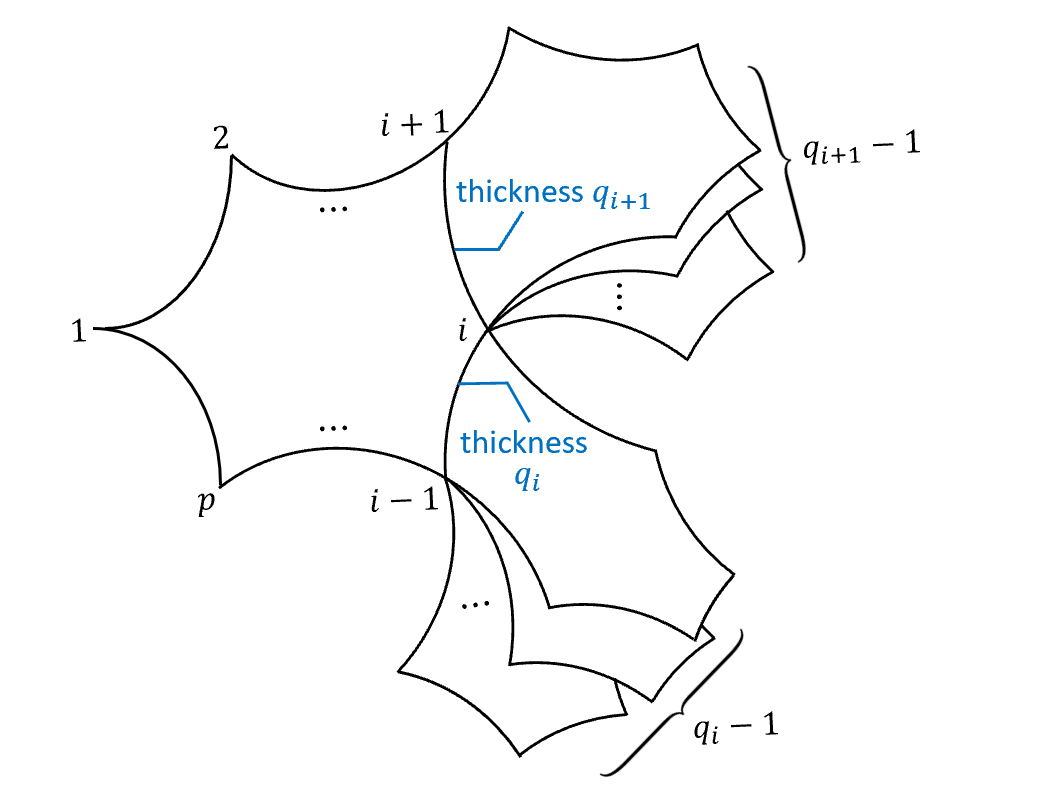}
    \caption{Right-angled Fuchsian building $I_{p, \mathbf{q}}$}
    \label{fig: example of Ipq}
\end{figure}

The developability of this complex of groups $G(P)$ has already been proven from Theorem \ref{thm_nonpositive developable}, thus the existence of $I_{p, \mathbf{q}}$ is guaranteed for every $p$ and $\mathbf{q}$. A right-angled hyperbolic building is clearly a labeled 2-dimensional polygonal complex with the following properties:
\begin{enumerate}
    \item $I_{p, \mathbf{q}}$ is contractible.
    \item Each polygon is a copy of the labeled complex $P$ with labels in edges vertices. (Orientation of the labeling can be reversed.)
    \item\label{properties_link} For a vertex $\sigma$ labeled by $\{i, i+1\}$, the link $Lk(\sigma, X)$ is the complete bipartite graph $K_{q_i, q_{i+1}}$.
\end{enumerate}

Futer-Thomas \cite{futer2012surface} shows the condition of local groups near $\sigma$ for the local development $\text{St}(\Tilde{\sigma})$ to have the link $L$ as a bipartite graph $K_{v,v}$.

\begin{lemma}[{Futer-Thomas \cite{futer2012surface}}]\label{lem_localconst}
    Let $X$ be a compact orientable hyperbolic surface tessellated by right-angled $p$-gons. Suppose that $G(X)$ is a complex of groups over $X$. Let $\sigma$ be a vertex in this tiling and let $G=V_{\sigma}$. Let the adjacent edge groups $E_j$ and face groups $F_k$ be as in Figure \ref{fig_lem local const} for $j, k=1, 2, 3, 4$. Since $G(X)$ is simple, we may identify each $E_j$ and $F_k$ with its image under inclusion into $V$. 

    Then the local development $\text{St}(\Tilde{\sigma})$ has link $L$ the complete bipartite graph $K_{v, v}$ if and only if 
    \begin{equation}
        V=E_1E_2=E_2E_3=E_3E_4=E_4E_1,
    \end{equation}
    \begin{equation}
        E_1\cap E_2=F_1, E_2\cap E_3=F_2, E_3\cap E_4=F_3, E_4\cap E_1 =F_4,
    \end{equation}
    and
    \begin{equation}
        |V:E_1|+|V:E_3|=v=|V:E_2|+|V:E_4|.
    \end{equation}
\end{lemma}
\begin{figure}
    \centering
    \includegraphics[width=0.4\linewidth]{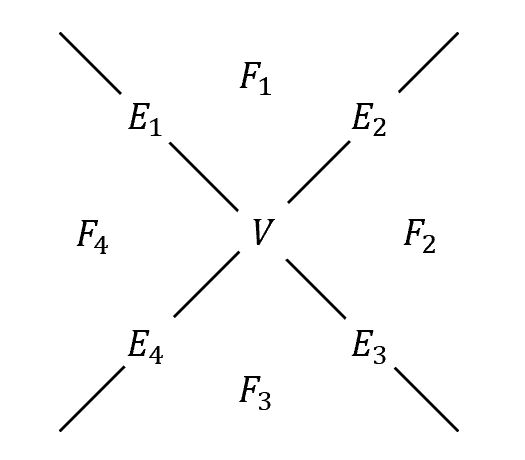}
    \caption{Local groups near $\sigma$}
    \label{fig_lem local const}
\end{figure}
For the proof, we refer the reader to \cite{futer2012surface}.

\subsection{Construction of Surface Tessellation}\label{subsec2.2}
Throughout the paper, compact surfaces are without boundaries. This subsection deals with the method of constructing a tessellation of a hyperbolic surface proposed in \cite{futer2012surface}. 
\\
Edmond-Ewing-Kulkarina \cite{edmonds1982regular} proved the following theorem regarding how many faces are needed for the compacted orientable surface of genus $g$ to be tessellated into a right-angled hyperbolic $p$-gon. 

\begin{lemma}[{Edmond-Ewing-Kulkarina \cite{edmonds1982regular}}]\label{lem_facenum}
    If a compact orientable surface of genus $g$ can be tessellated with $F$ right-angled $p$-gons, the value of $F$ is given by:
    \begin{equation}\label{eq_numface}
        F=\displaystyle\frac{8(g-1)}{p-4}
    \end{equation}
\end{lemma}

\begin{proof}[Sketch of the proof]
    The proof is based on the computation using the Gauss-Bonnet theorem:
    \begin{equation}\label{eq_GaussBonnet}
        \int_M{KdA}+\int_{\partial M}{k_g}ds=2\pi \chi(M)
    \end{equation}
    ($K$ and $k_g$ denote Gaussian curvature of $M$ and the geodesic curvature of $\partial M$, respectively.)

    Using equation \ref{eq_GaussBonnet} for a piece of regular right-angled $p$-gon yields the area of a $p$-gon is $\frac{\pi}{2}(p-4)$. Using \ref{eq_GaussBonnet} again for the entire tessellated surface, we get 
    \begin{equation}\label{eq_GB2}
        -F(\text{area of a }p\text{-gon})=2\pi(2-2g).
    \end{equation}
    Equation \ref{eq_numface} follows immediately from \ref{eq_GB2}.
\end{proof}

Futer and Thomas' paper \cite{futer2012surface} presents a method for tessellating a compact orientable surface of genus $g$ using $F$ distinct right-angled hyperbolic $p$-gons. Their tessellation method is divided into two cases where $F$ is divisible by 4 or not. 
\begin{figure}
    \centering
    \includegraphics[width=0.8\textwidth]{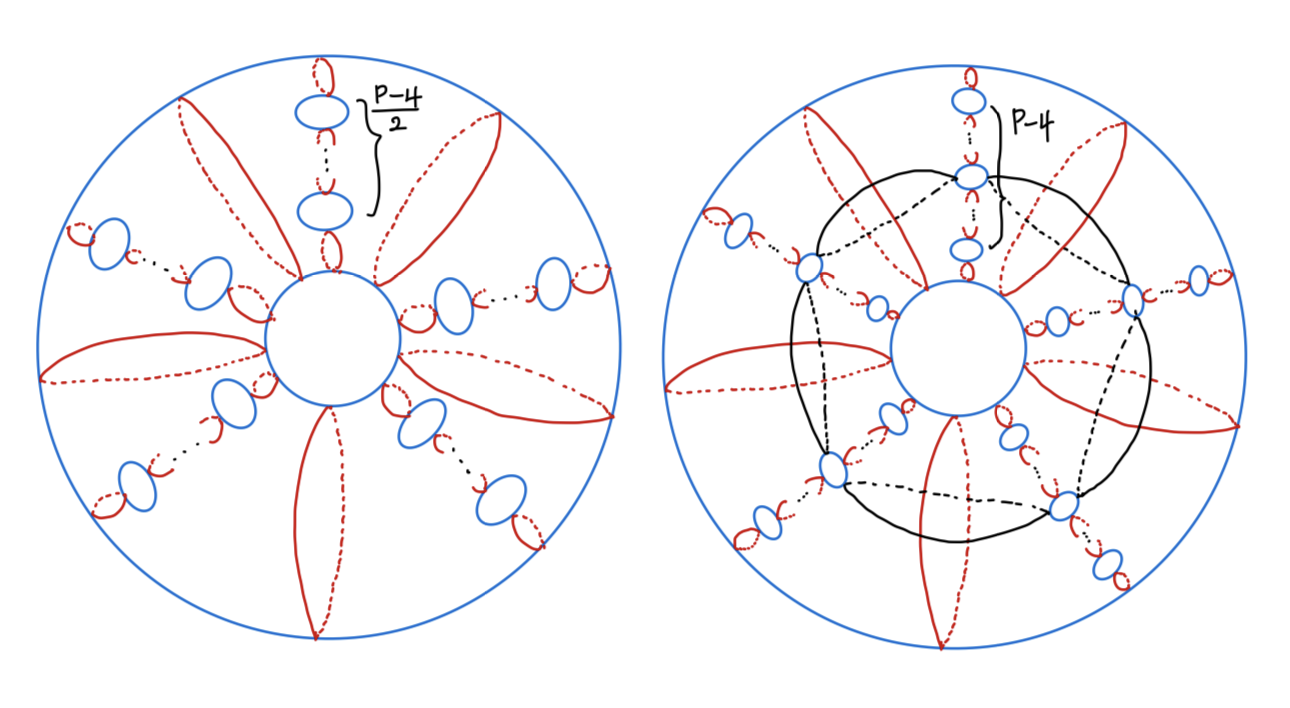}
    \caption{left: $p$ is even, right: $p$ is odd}
    \label{fig: tessellation when F is divisible by 4}
\end{figure}

If $F$ is divisible by 4, the tessellation can be visualized clearly as shown in Figure \ref{fig: tessellation when F is divisible by 4}. The left-hand side of Figure \ref{fig: tessellation when F is divisible by 4} illustrates the case where $F$ is divisible by 4 and $p$ is even. In this case, four pieces of $p$-gons colored in Figure \ref{fig: tessellation when F is divisible by 4} generate one \textit{block}, which is a surface of genus $\frac{p-4}{2}$ with two boundary circles. The right-hand side of Figure \ref{fig: tessellation when F is divisible by 4} depicts the case where $F$ is divisible by 4 and $p$ is odd. Since $F=\frac{8(g-1)}{p-4}$, $F$ is divisible by 8 if $p$ is odd. In this case, eight pieces of $p$-gons colored in Figure \ref{fig: tessellation when F is divisible by 4} are considered as one \textit{block}, which is a surface of genus $p-4$ with two boundary circles (See Figure \ref{fig: tessellation when F is divisible by 4}). By connecting $\frac{F}{8}$ blocks in a circular shape, a genus $g$ hyperbolic surface is obtained.

\begin{figure}
    \centering
    \includegraphics[width=0.8\textwidth]{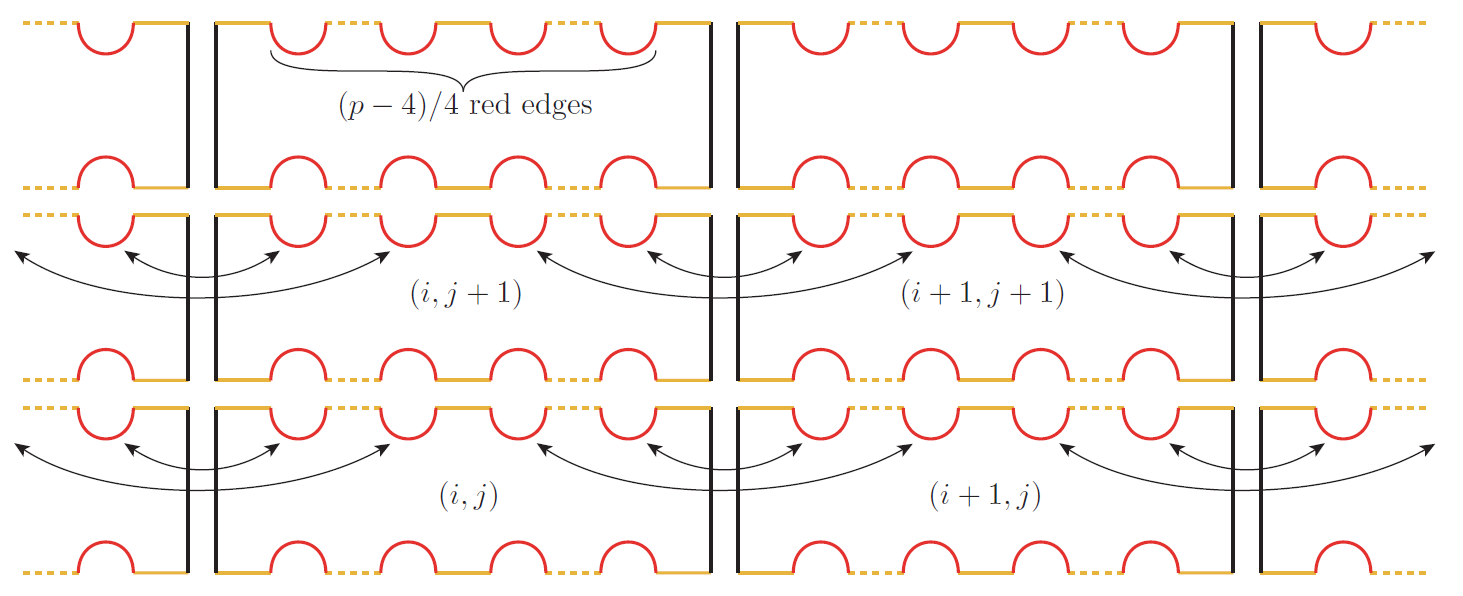}
    \caption{tessellation when $F$ is not divisible by 4 (Figure from \cite{futer2012surface})}
    \label{fig: tessellation when F is not divisible by 4}
\end{figure}
For the case where $F$ is not divisible by 4, $F$ is a composite number. Remark that \cite{futer2012surface} is not proposing a tessellation method when $F$ is prime. Consider the situation where $F$ many $p$-gons are arranged rectangularly, as shown in Figure \ref{fig: tessellation when F is not divisible by 4}. Let $F=ab$, and label the position from $(1, 1)$ to $(a, b)$ according to the $p$-gon's location in the $a \times b$ rectangular arrangement. A tessellated genus $g$ hyperbolic surface can be obtained by the following steps:
\begin{enumerate}
    \item Attach the top side and bottom side of the rectangular arrangement. Similarly, attach the left and right sides of the rectangular arrangement.This forms a torus with $\frac{ab(p-4)}{4}$ holes.
    \item Pair the holes as shown in the figure \ref{fig: tessellation when F is not divisible by 4} and glue them together. This results in a genus $g$ surface.
\end{enumerate}

Futer and Thomas's paper \cite{futer2012surface} provides more detailed information about these tessellation methods. In this paper, we will only focus on the parity of the length of the loops formed by the quotient of geodesics in the original building, omitting the detailed construction process.

\section{Homology group of tessellated surface}\label{sec3}

In this section, we prove several properties of the first homology group of a hyperbolic tessellated compact surface before proving the main result in Section \ref{sec4}.  A \emph{loop} refers to a simple closed curve on the surface.

We begin with the most general property, which holds for any compact surface $X$ without boundary, regardless of whether $X$ is tessellated.
Theorem \ref{thm_genloopbasis} states that the closed loops satisfying certain conditions generate the first homology group of a surface. This result is of independent interest beyond its application to the right-angled tessellated surface addressed in this paper.

\begin{theorem}\label{thm_genloopbasis}
    Let $X$ be a 2-dimensional compact orientable surface, and let $S_1, S_2, \cdots, S_k$ be homeomorphic images of $S^1$ on $X$ such that $(X, S_1 \cup S_2 \cup \cdots \cup S_k)$ is a good pair. The intersection of $S_i$ and $S_j$($i\neq j$) are either empty or consists of a single point, and the intersection of $S_i$, $S_j$, and $S_k$ is empty for any distinct integers $i, j, k$. If $X/{S_1 \cup S_2 \cup \cdots \cup S_k}$ is orientable and has a trivial first homology group, then $[S_1], [S_2], \cdots, [S_k]$ generate the first homology group of $X$.
\end{theorem}

\begin{proof}
    The proof is based on the generalized Mayer-Vietoris principle, which can be found in Chapter 8 of \cite{bott2013differential}. From now on, let $A$ denote $S_1 \cup S_2 \cup \cdots \cup S_k$.
    
    First, consider a sufficiently small neighborhood $U_i$ of $S_i$ in $X$, where $U_i \cap U_j \cap U_k$ is empty for any distinct integers $i, j, k$. Let $\iota_i : U_i \hookrightarrow \cup_{i}U_i$ be an inclusion map for each $i$, and define $\iota=\iota_1 \oplus \iota_2 \oplus \cdots \oplus \iota_k$.
    
    We have 
    \begin{equation}\label{eq_genMV}
        \cdots \rightarrow\bigoplus_{i<j}{\Tilde{H}_1(U_i \cap U_j)}\rightarrow \bigoplus_{i}{\Tilde{H}_1(U_i)}\xrightarrow[]{\iota_{*}}\Tilde{H}_1(A)\rightarrow\bigoplus_{i<j, U_i \cap U_j \neq \emptyset}\Tilde{H}_0(U_i \cap U_j)\rightarrow \cdots 0.
    \end{equation}
    since there are no triple intersections among the open cover $\{ U_i \}$ of $S_1 \cup S_2 \cup \cdots \cup S_k$, the map $\iota_*$ is an epimorphism. This holds because each $\Tilde{H}_0(U_i \cap U_j)$ is trivial for every $i, j$ such that $U_i \cap U_j$ is nonempty. For each $i$, $\Tilde{H}_{1}(U_i)$ represents the homology class $[S_i]$, and hence $[S_1], [S_2], \cdots, [S_k]$ generate $\Tilde{H}_{1}(A)$.

    Given the assumption that $(X, A)$ is a good pair, the long exact sequence for this pair is as follows:
    \begin{equation}\label{eq_les}
        \cdots \rightarrow \Tilde{H}_2(X/A)\rightarrow \Tilde{H}_1(A)\xrightarrow[]{\mathit{i}_*}\Tilde{H}_1(X)\xrightarrow[]{\mathit{j}_*}\Tilde{H}_1(X/A)\rightarrow\cdots 
    \end{equation}
    where $\mathit{i}$ is the inclusion map $A\hookrightarrow X$, and $\mathit{j}$ is the quotient map $X\rightarrow X/A$. Since $\Tilde{H}_1(X/A)$ is trivial, the induced map $\mathit{i}_*$ is an epimorphism. Consequently, the images of $S_i$ under the map $\mathit{i}\circ\iota_{i}$ generate $\Tilde{H}_1(X) =H_1(X)$. 
\end{proof}

Now we need to specify Theorem \ref{thm_genloopbasis} to obtain results about right-angled tessellated hyperbolic surface $X$. Before stating Corollary \ref{cor_loopbasis}, which will be used in Section \ref{sec4}, the classification of the loops must first be defined. In the following definitions, let $X$ be the tessellated surface. 

\begin{defn}[boundary loop]\label{def_bdryloop}
     A closed curve $l$ on $X$ composed of the $n$ distinct edges of polygons is called a \emph{boundary loop of length $n$}.
\end{defn}

\begin{defn}[polygonal loop]\label{def_bbl}
    A closed curve $l$ is called a \emph{polygonal loop} if it is a boundary of a single $p$-gon.
\end{defn}

\begin{defn}[boundary geodesic loop]\label{def_bgl}
    A closed curve $l$ is called a \emph{boundary geodesic loop} if it is both a boundary loop (as defined in Definition \ref{def_bdryloop}) and a closed geodesic loop.
\end{defn}
For quotients of Bourdon's building $I_{p, \mathbf{q}}$, a boundary geodesic loop can be defined as a loop that is a quotient of a boundary geodesic on $I_{p, \mathbf{q}}$.

However, if $X$ is a hyperbolic surface tessellated by right-angled polygons, it is clear that the intersection of more than three boundary geodesic loops is empty. Thus, if the given right-angled tessellation $X$ has the property that the intersection of any two boundary geodesic loops is either empty or consists of a single point, then $X$ automatically satisfies all the conditions required for Theorem \ref{thm_genloopbasis} to hold.

\begin{corollary}\label{cor_loopbasis}
    Let $X$ be a 2-dimensional hyperbolic compact surface tessellated by right-angled regular $p$-gons. If the intersection of any two boundary geodesic loops on $X$ is either empty or consists of a single point, than the set of boundary geodesic loops on $X$ generates the first homology group of $X$.
\end{corollary}

\begin{proof}
    Suppose that $X$ is a genus $g$ compact hyperbolic surface tessellated by $F$ right-angled regular $p$-gons. Let $S_1, S_2, \cdots, S_k$ denote all the basic boundary loops contained in $X$, and let $A=S_1 \cup S_2 \cup \cdots \cup S_k$, a subspace of $X$. Since $A$ is the union of all the edges contained in $X$, it follows that $X/A$ is homeomorphic to the wedge of $F$ numbers of $S^2$. It is evident that $(X, A)$ is a good pair and that $\Tilde{H}_1(X/A)$ is trivial. Furthermore, the remaining conditions required for Theorem \ref{thm_genloopbasis} have already been assumed and verified in the arguments above. Therefore, by Theorem \ref{thm_genloopbasis}, it follows that $S_1, S_2, \cdot, S_k$ generate $H_1(X)$.
\end{proof}

The tessellated surface $X$ that satisfies the conditions of Corollary \ref{cor_loopbasis} will serve as the base complex for the quotient complex of groups $I_{p, \mathbf{q}}/\Gamma_{p, \mathbf{q}, g}$ is Section \ref{sec4}. It should be noted that the regularity of the $p$-gon is not used in the proof of Corollary \ref{cor_loopbasis}. In Section \ref{sec4}, the result of Corollary \ref{cor_loopbasis} is extended to the divided right-angled regular tessellation.

\section{Existence of $\Gamma_{p, \mathbf{q}, g}$}\label{sec4}
In this section, we find the sufficient conditions under which a lattice $\Gamma_{p, \mathbf{q}, g}$ exists for a given $\mathbf{q}$. We begin by generalizing Futer and Thomas \cite{futer2012surface}.

\begin{prop}\label{prop_genFT}
    In the same context as in Lemma \ref{lem_localconst}, let us assume, without loss of generality, that the index $i$ in the edge group $E_{i}$ increases counterclockwise along the boundary of each $p$-gon. The local development $St(\Tilde{\sigma})$ has a link $L$ as the complete bipartite graph $K_{q_{i}, q_{i+1}}$ with $i$ increasing sequentially counterclockwise if and only if
    \begin{equation}\label{eq4.1}
    V=E_{1}E_{2}=E_{2}E_{3}=E_{3}E_{4}=E_{4}E_{1}
    \end{equation}
    \begin{equation}\label{eq4.2}
        E_{1}\cap E_{2}=F_{1}\text{, }E_{2}\cap E_{3}=F_{2}\text{, } E_{3}\cap E_{4}=F_{3}\text{, }E_{4}\cap E_{1}=F_{4}
    \end{equation}
    and
    \begin{equation}\label{eq4.3}
        |V:E_{1}|+|V:E_{3}|=q_{i}\text{, }|V:E_{2}|+|V:E_{4}|=q_{i+1}
    \end{equation}
\end{prop}

In the above Lemma, $E_i$ and $F_j$ are considered as subgroups of $V$. The proof is verbatim the same as the proof of lemma \ref{lem_localconst}.

Lemma \ref{lem_difference of homotopic loop} shows that, in a polygonal tessellated surface, the difference between two homotopic loops can be expressed as the sum of several polygonal loops.

\begin{lemma}\label{lem_difference of homotopic loop}
    Let $X$ be a compact orientable surface tessellated by polygons, and let $l$ and $k$ be boundary loops of $X$. If $[l_0]$ and $[l_1]$ represent the same homology class in $H_1(X)$, then there exist polygonal loops $k_1, k_2, \cdots, k_n$ such that $l_1=l_0+\displaystyle\sum_{i=1}^{n}{k_i}$. (Throughout this paper, summation between loops means the composition of loops, while disregarding the orientation.)
\end{lemma}
\begin{proof}
    If the intersection of $l_0$ and $l_1$ is empty, cutting $X$ along $l_0$ and $l_1$ yields a submanifold with boundary $R$ that is homeomorphic to $S^1 \times [0, 1]$ (See \cite{marden1966regions}). In this case, it is evident that $R$ is a union of several polygons, and $l_0$ can be expressed as the composition of $l_1$ and the polygonal loops corresponding to those polygons.

    If the intersection of $l_0$ and $l_1$ is nonempty, assume there are $n$ intersection points. In this case, it is possible to construct boundary loops $l_0'$ and $l_1'$ such that $l_0 \cup l_1 = l_0' \cup l_1'$ and the loops do not cross each other. By applying the same arguments as in \cite{marden1966regions}, the region obtained by cutting $X$ along $l_0$ and $l_1$ is homeomorphic to $S^1 \times [0, 1]/~$, where $(x_i, 0)~(x_i, t)$ for $t \in [0, 1]$ and $i=1, 2, \cdots, n$ (with each $x_i$ being a distinct point on $S^1$). It is clear that this region is also a union of polygons, allowing us to reach the same conclusion as in the previous case.

\end{proof}

Before proving the main theorem, Theorem \ref{prop_mainthm}, we need to prove the possibility of providing an orientation to the tessellated surface, as stated in Lemma \ref{lem_orientation}.

\begin{defn}[dual graph of tessellated surface]\label{def_dualgraph}
    Let $X$ be a 2-dimensional compact surface tessellated by polygons. The \emph{dual graph} of $X$, denoted by $\text{Dual}(X)$, is defined as a graph $G=(V, E)$ determined by the following conditions:
    \begin{enumerate}
        \item Each vertex in $V(G)$ corresponds to a face of $X$.
        \item For any distinct vertices $\sigma, \tau \in V(G)$, the edge connecting $\sigma$ and $\tau$ is an element of $E(G)$ if and only if the face corresponding to $\sigma$ and the face corresponding to $\tau$ share an edge on $X$.
    \end{enumerate}
\end{defn}

\begin{lemma}\label{lem_orientation}
    Let $X$ be a 2-dimensional compact hyperbolic surface tessellated by right-angled $p$-gons, such that the intersection of any two boundary geodesic loops on $X$ is either empty or consists of a single point. If every boundary geodesic loop in $X$ has an even length, then it is possible to assign an orientation (counterclockwise or clockwise) to each $p$-gon, ensuring that any two adjacent $p$-gons have opposite orientations.
\end{lemma}

\begin{proof}
    For a given tessellated surface $X$, the dual graph $\text{Dual}(X)$ can be drawn on $X$. Note that this can be interpreted as a tessellation of $X$ by 4-gons, as each vertex of $X$ is connected with four edges. Each 4-gon in $\text{Dual}(X)$ corresponds to a vertex of $X$ it contains.

    An orientation can now be assigned to each face of $X$ through the following process: First, choose any face arbitrarily (denote this face as $F_0$) and assign it a counterclockwise orientation. Then, for any other face $F_1$, select a path on $\text{Dual}(X)$ that connects the vertices corresponding to $F_0$ and $F_1$. If the path has an even length, assign a counterclockwise orientation to $F_1$; if the path has an odd length, assign a clockwise orientation to $F_1$.

    This method of assigning orientations satisfies the given condition that any two adjacent $p$-gons have opposite orientations since adjacent faces are connected by a single edge in $\text{Dual}(X)$. The only remaining thing to prove is that this method of assigning orientations is well-defined—that is, any path connecting the same two vertices has the same length mod 2. This is equivalent to proving that every cycle has an even length, so it suffices to show that each cycle in $\text{Dual}(X)$ has an even length.

    Let $X^{*}$ be a 1-dimensional simplicial complex corresponding to $\text{Dual}(X)$, which can be considered as an embedding into $X$. Let $\imath: X^{*}\hookrightarrow X$ denote the inclusion map of $X^{*}$ into $X$, and let $\imath_{*}:H_1(X^{*})\rightarrow H_1(X)$ be the group homomorphism induced by $\imath$. 
    
    Let $l_1, l_2, \cdots, l_k$ be the boundary geodesic loops in $X$. For each boundary geodesic loop $l_i$, there are two corresponding cycles in the graph $\text{Dual}(X)$ (see the red cycles in Figure \ref{fig_cycle in dualgraph}). Choose one of these two cycles arbitrarily and denote it as $C_i$ for each $i=1, 2, \cdots, k$. Clearly, $C_i$ is homotopic to $l_i$ when viewed as a closed loop on $X$. By Corollary \ref{cor_loopbasis}, the loops $l_1, l_2, \cdots, l_k$ generate the first homology group of $X$, and thus $C_1, C_2, \cdots, C_n$ do as well. 
    
    Therefore, for any loop $l$ on $X^{*}$, $[l]\in H_1(X^{*})$ can be expressed as $[l]=\sum_{i=1}^{n}{k_i [C_i]}$ for some integers $k_i$. It follows from Lemma \ref{lem_difference of homotopic loop} that $l-\sum_{i=1}^{n}{k_i C_i}\in \ker (\imath)$, as $X^{*}$ can be regarded as a 4-gon tessellation on $X$. Thus, $l$ can be written as the sum of $\sum_{i=1}^{n}{k_i C_i}$ and several 4-gons. Since each $l_i$ has an even length, it follows that $C_i$ also has an even length, and thus every cycle in $\text{Dual}(X)$ has an even length.
    \begin{figure}
        \centering
        \includegraphics[width=0.5\linewidth]{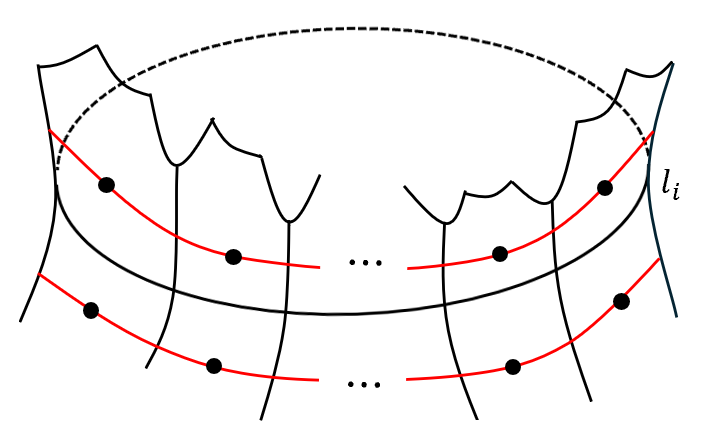}
        \caption{The red cycles are the cycles on the dual graph corresponding to the boundary geodesic loop $l_i$. This figure is drawn for the case $p=6$, but note that the specific value of $p$ is not crucial for this figure.}
        \label{fig_cycle in dualgraph}
    \end{figure}
\end{proof}

Proposition \ref{prop_mainthm} will be proven using the fact that if $X$ satisfies a condition of Corollary \ref{cor_loopbasis} (i.e. the intersection of any two boundary geodesic loops is either empty or consists of a single point), then for every boundary loop in $X$, there exists a composition of several boundary geodesic loops that is homotopic to $X$. This fact is also used in the proof of Lemma \ref{lem_orientation}.

The following outline summarizes the main steps of the proof of Proposition \ref{prop_mainthm}: 
First, we color the edges of the given tessellated complex using two colors, 0 and 1. It is shown that if there exists an edge coloring satisfying certain conditions \ref{item_altercond} and \ref{item_consistcond}, referred to as a \emph{good coloring}, it is possible to assign edge groups and vertex groups that satisfy the conditions \ref{eq4.1}, \ref{eq4.2}, and \ref{eq4.3} in Proposition \ref{prop_genFT}, based on the color of the edges. A canonical method for generating a good coloring is proposed, where colors are assigned along a path, and this coloring is well-defined if this method is independent of the choice of path. Through a slightly technical approach, this path independence is reduced to a statement about the boundary loop. An inductive process on the boundary loop is used to complete the proof; starting with the boundary geodesic loop, the result is extended to any boundary loop by proving that the desired property is preserved under loop composition.

\begin{prop}\label{prop_mainthm}
    Given an even integer $p$ and an integer $g$ such that $\frac{8(g-1)}{p-4}$ is an integer, and an alternating non-coprime sequence $\mathbf{q}=\{ q_1 , q_2 , \cdots , q_p \}$, consider a compact orientable surface $S_g$ with genus $g$. If there exists a tessellation of $S_g$ composed of right-angled $p$-gons such that every boundary geodesic loop has an even length and the intersection of any two boundary geodesic loops is either empty or consists of a single point, then a surface quotient lattice $\Gamma_{p, \mathbf{q}, g}$ exists.
\end{prop}

\begin{proof}
    Before starting the proof, let $\mathbf{q}=\{ q_1 , q_2 , \cdots , q_p \}$, with $d\geq 2$ being the greatest common divisor of $q_1 , q_3 , \cdots , q_{p-1}$, and $e\geq 2$ being the greatest common divisor of $q_2 , q_4 , \cdots , q_p$. Set $q_i=d{q_i}'$ for odd indices $i$, and $q_i=e{q_i}'$ for even indices $i$.

    Let $T(S_g)$ be a polygonal complex obtained from a $p$-gon tessellation of $S_g$. Color the edges of $T(S_g)$ using the colors 0 and 1, denoted by
    \begin{center}
        $c: E(T(S_g))\rightarrow \{ 0, 1 \}$.
    \end{center}

    If it is possible to color all edges while satisfying the following conditions, this coloring is referred to as a \textit{good coloring}.
    \begin{enumerate}
        \item\label{item_altercond} \textit{Alternating condition}: The edges included in each boundary geodesic loop must be assigned 0 and 1 alternately.
        \item\label{item_consistcond} \textit{Consistency condition}: For a boundary geodesic loop $l$ with a given orientation, all colors of the edges connected to $l$ and placed on the right(respectively, left) side of $l$ must be same.
    \end{enumerate}

    Assume that a good coloring is assigned to the edges of $T(S_g)$. Then, assign a group to each edge based on the color assigned to that edge. Let $e$ be a type $i$ edge on $T(S_g)$. If $i$ is odd, assign $\mathbb{Z}_{d-1}\times {\mathbb{Z}_{e-1}}^{c(e)}\times\mathbb{Z}_{{q_i}'}$. If $i$ is even, assign ${\mathbb{Z}_{d-1}}^{c(e)}\times \mathbb{Z}_{e-1}\times\mathbb{Z}_{{q_i}'}$. Then, assign $\mathbb{Z}_{d-1}\times\mathbb{Z}_{e-1}\times\mathbb{Z}_{{q_i}'}\times\mathbb{Z}_{{q_{i+1}}'}$ to the vertex $\sigma$, which is located on the intersection of type $i$ geodesic and type $i+1$ geodesic. Now, we need to assign a face group to each face of $T(S_g)$ to obtain a complex of groups. For a face $f$ on $T(S_g)$, select two adjacent edges $e_1$ and $e_2$ on the boundary of $f$, which are assigned groups $E_1$ and $E_2$, respectively, and assign to $f$ the intersection $\iota E_1\cap \iota E_2$ of the images of $E_1$, $E_2$ in the vertex group. This assignment is independent of the choice of $e_1$ and $e_2$ because of the consistency condition \ref{item_consistcond}. Furthermore, for any vertex $\sigma$ on $T(S_g)$, the local group near $\sigma$ satisfies the conditions of Proposition \ref{prop_genFT} because of the alternating condition \ref{item_altercond} (See Figure \ref{fig_localconst}). Thus, the complex of groups on $T(S_g)$ has $I_{p, \mathbf{q}}$ as its universal cover.
    
    \begin{figure}
        \centering
        \includegraphics[width=0.4\textwidth]{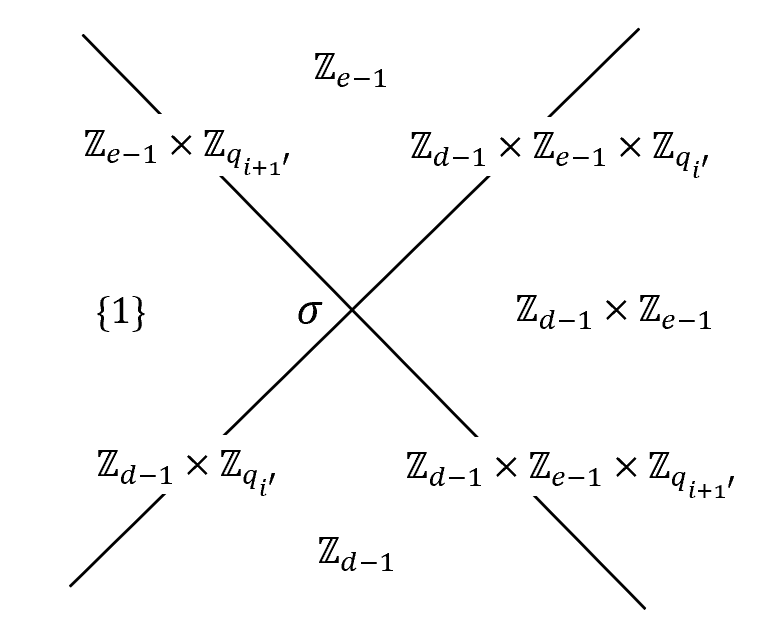}
        \caption{edge and face groups near $\sigma$}
        \label{fig_localconst}
    \end{figure}
    
    Now, it remains to prove the existence of a \textit{good coloring} if every boundary loop on $T(S_g)$ has an even length. Select an arbitrary vertex $\sigma_0$ on $T(S_g)$ and define it as the \textit{base point}. Denote the four edges connected to $\sigma_0$ as $e_1$, $e_2$, $e_3$, and $e_4$ in counterclockwise order. (In other words, $e_1$ and $e_3$ share the same geodesic type, and $e_2$ and $e_4$ share the same geodesic type.) Define the values of the coloring function $c$ for these four edges as $c(e_1)=0$, $c(e_2)=0$, $c(e_3)=1$, and $c(e_4)=1$. Then, the remaining edges on $T(S_g)$ can be colored as follows: for another vertex $\sigma$ on $T(S_g)$, it is possible to select a path composed of edges of $T(S_g)$ starting from $\sigma_0$ and ending at $\sigma$. Then the color $c(e)$ can be uniquely determined so that the alternating condition \ref{item_altercond} and the consistency condition \ref{item_consistcond} are satisfied.
    
    We need to show that $c(e)$ is independent of the choice of path connecting $\sigma_0$ and $\sigma$. For directed path $l_1$ and $l_2$, we define functions $\phi : \{0, 1\}\times \{0, 1\} \rightarrow \{0, 1\}\times \{0, 1\}$ and $\psi : \{0, 1\}\times \{0, 1\} \rightarrow \{0, 1\}\times \{0, 1\}$. As shown as Figure \ref{fig_comploop}, let $a$ be a color of edge on $l_1$ connected to $\sigma_0$ and $b$ be a color of edge connected to $\sigma_0$ such that $b$ is left on $a$. Then, $\phi_1(a, b)$ and $\phi_2(a, b)$ are defined as the colors of the edges connected to $\sigma$ determined along $l_1$, while satisfying the alternating condition \ref{item_altercond} and consistency condition \ref{item_consistcond}. This function $\phi(a, b)$ is defined as $\phi(a, b)=(\phi_1(a, b), \phi_2(a, b))$. For $a \in \{0, 1\}$, define $\bar{a}$ as
    \begin{equation}
        \bar{a}=\begin{cases}
            1 & \text{a=0} \\ 0 & \text{a=1}
        \end{cases}.
    \end{equation}
    Then, $\psi(\bar{a}, \bar{b})=(\psi_1(\bar{a}, \bar{b}), \psi_2(\bar{a}, \bar{b}))$ can be similarly defined along the path $l_2$ (See Figure \ref{fig_comploop}). To prove the color of edges connected to $\sigma$ determined independently of the choice of path, it is enough to show that $\psi^{-1}(\phi(a, b))=(a, b)$ for any $(a, b)\in \{0, 1\} \times \{0, 1\}$.
    \begin{figure}
        \centering
        \includegraphics[width=0.5\textwidth]{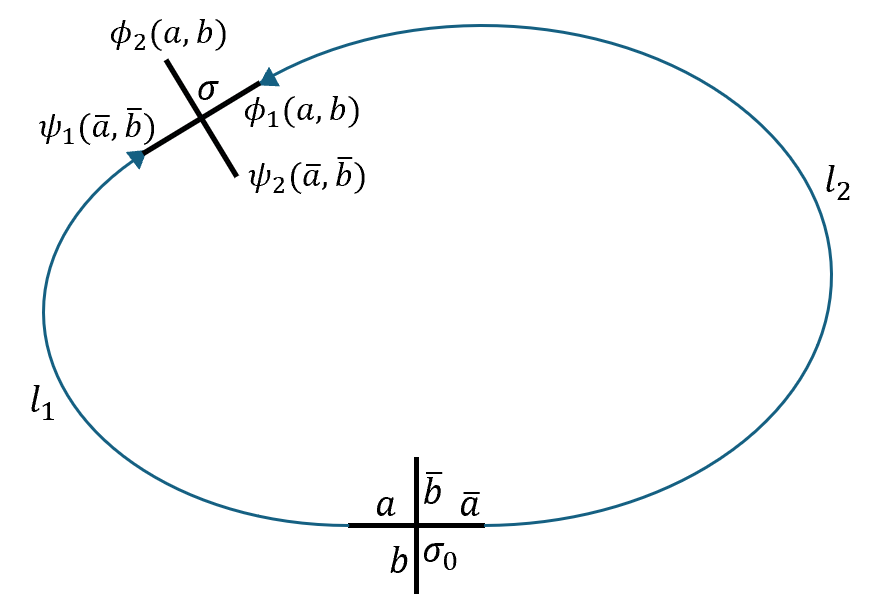}
        \caption{definition of $\phi$ and $\psi$}
        \label{fig_comploop}
    \end{figure}
    
    For any path $C$, let $C^{-1}$ denote the path obtained by reversing the orientation of $C$. For two paths $C_1$ and $C_2$ where the endpoint of $C_1$ is the initial point of $C_2$, denote $C_1 \# C_2$ as the concatenation of the two paths $C_1$ and $C_2$, sharing the initial point with $C_1$ and the endpoint with $C_2$. For two paths $l_1$ and $l_2$ as shown in Figure \ref{fig_comploop}, let $l=l_1 \# {l_2}^{-1}$. Then $l$ forms a loop containing the base point $\sigma_0$. Define $\Phi_l : \{0, 1\} \times \{0, 1\} \rightarrow\{0, 1\} \times \{0, 1\}$ as $\psi^{-1}\circ\phi$ for $\psi$ and $\phi$ defined by $l_1$ and $l_2$. Then it is enough to show that for every loop $l$ pass $\sigma_0$, $\Phi_l$ is the identity map. 

    We start the proof by verifying that $\Phi_l$ is an identity map for the case where $l$ is a polygonal loop (Definition \ref{def_bbl}) or a boundary geodesic loop (Definition \ref{def_bgl}). If $l$ is a polygonal loop, $\Phi_l$ is an identity map since $l$ has length $p$ and $p$ is even. Otherwise, if $l$ is a boundary geodesic loop, it follows that $\Phi_l$ is a boundary loop because of the alternating condition \ref{item_altercond} and the assumption that every boundary geodesic loop has an even length.

    Because of Theorem \ref{thm_genloopbasis}, for any nontrivial boundary loop $l$ on $T(S_g)$ containing $\sigma_0$, there exists a loop $l'$ on $T(S_g)$ that is homotopic to $l$ and can be obtained by composing the boundary geodesic loops. The boundary loop $l'$ can be expressed as $\sum{k_i l_i}$ for positive integer $k_i$ and boundary geodesic loop $l_i$. Then it follows that $l$ can be expressed as a sum of boundary geodesic loops $\sum{k_i l_i}$ plus some polygonal loops from Lemma \ref{lem_difference of homotopic loop}. If $l$ is trivial, the fact that a trivial closed loop on $S_g$ bounds a closed disc on $S_g$ \cite{marden1966regions} implies that $l$ is a sum of some closed boundary loops.
\end{proof}

\begin{corollary}\label{cor_exist_1}
    Let $p$ be an even integer and $g$ be an integer such that $F=\displaystyle\frac{8(g-1)}{p-4}$ is an integer divisible by 4. Then, for every alternating non-coprime sequence $\mathbf{q}=\{q_1, q_2, \cdots, q_p\}$, a lattice $\Gamma_{p, \mathbf{q}, g}$ exists.
\end{corollary}

\begin{proof}
    In the tessellation presented on the left side of Figure \ref{fig: tessellation when F is divisible by 4}, every boundary geodesic loop has an even length. Additionally, there are no pairs of boundary geodesic loops whose intersection contains more than two vertices. The result follows directly from Proposition \ref{prop_mainthm}.
\end{proof}

We also obtain $\Gamma_{p, \mathbf{q}, g}$ in the case where $F$ is not divisible by 4 by dividing the $p$-gon into 2 or 4 pieces. When the $p$-gon is divided along its symmetry axis, the condition of Proposition \ref{thm_genloopbasis} is still satisfied. Since the proof of Proposition \ref{prop_mainthm} does not depend on metric information, it remains applicable if the conditions of Corollary \ref{cor_loopbasis} (boundary geodesic loops generate the basis of the first homology group) and the condition regarding the length of boundary loop (every boundary geodesic loop has an even length) are satisfied when each $p$-gon is divided into smaller pieces.

\begin{corollary}\label{cor_exist_2}
    Let $p$ be an even integer and $g$ be an integer such that $F=\displaystyle\frac{8(g-1)}{p-4}$ is an even integer, however, not divisible by 4. Then, if length $p$ sequence $\mathbf{q}$ is 2-symmetric by $m$ and gcd of $q_{m-1}, q_{m+1}, \cdots, q_{m+p-1}$ is even, a lattice $\Gamma_{p, \mathbf{q}, g}$ exists.
\end{corollary}

\begin{proof}
    From the tessellation method presented in Figure \ref{fig: tessellation when F is not divisible by 4}, there are only two boundary loops that have an odd length. Divide each $p$-gon into two pieces as shown in Figure \ref{fig: divide 2 piece}. Then, consider the resulted tessellation as comprised of 2$F$ faces. Each face is a right-angled polygon with $p'=\frac{p+4}{2}$ edges, and the lengths of the edges do not need to coincide. (In Figure \ref{fig: divide 2 piece}, edge $PQ$ has a different length from the others contained in the original $p$-gon.) However, the resulting tessellation includes every boundary geodesic loop from the original tessellation, so it satisfies the condition of Proposition \ref{thm_genloopbasis} automatically. Additionally, every boundary geodesic loop has an even length since the new edge divides the edges (red edges in Figure \ref{fig: divide 2 piece}) contained in the boundary geodesic loop with odd length into two pieces.  
    \begin{figure}
        \centering
        \includegraphics[width=0.6\linewidth]{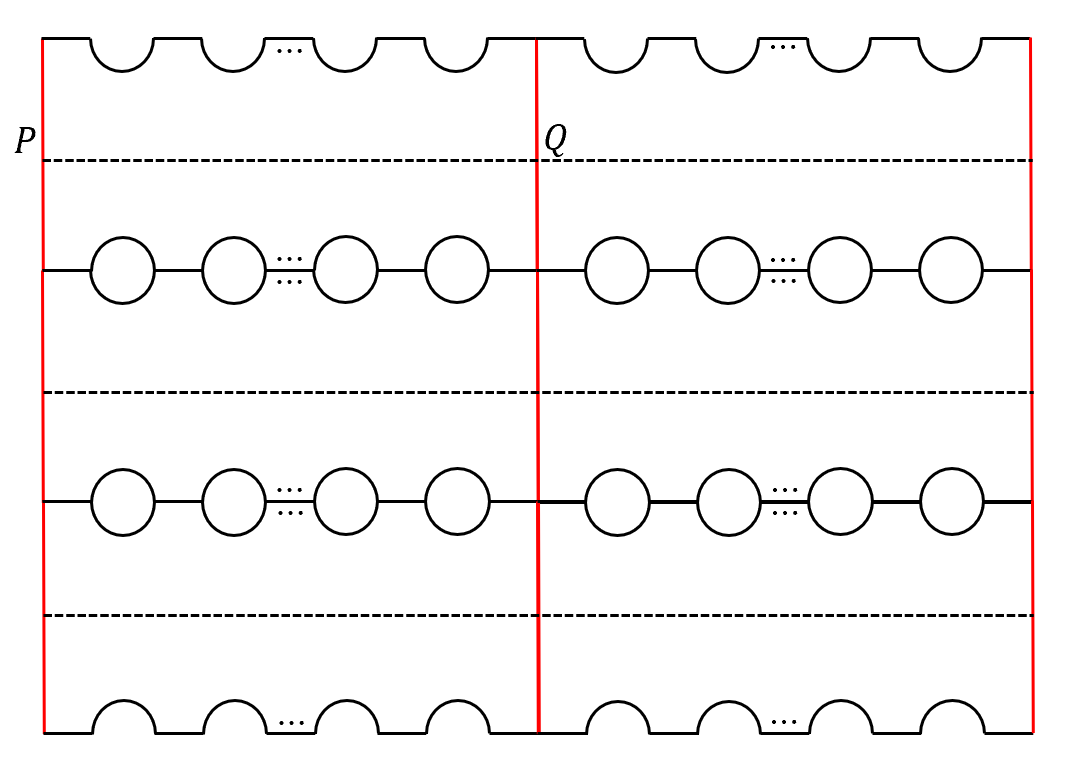}
        \caption{$F=6$ case: number of the row of $p$-gon is $\frac{F}{2}$}
        \label{fig: divide 2 piece}
    \end{figure}
    This means that the resulting tessellation by $2F$ $p'$-gons satisfies three conditions required for Proposition \ref{prop_mainthm}: first, the boundary loop must generate the first homology group of the surface, and second, every boundary loop must consist of an even number of edges, and third, the intersection of any two boundary geodesic loops is either empty or consists a single point (this can be easily verified from Figure \ref{fig: divide 2 piece}). Since the exact length of the loop does not affect the proof of Proposition \ref{prop_mainthm}, a lattice $\Gamma_{p', \mathbf{q'}, g}$ exists for any alternating non-coprime sequence $\mathbf{q'}$ with length $p'$.
    
    Without loss of generality, let $\mathbf{q}=\{q_1, q_2, \cdots, q_p\}$ be 2-symmetric by $m=1$. By assumption, $q_2, q_4, \cdots, q_p$, which are the elements of $\mathbf{q}$ with even indices are all even. Let the length $p'$ sequence $\mathbf{q'}$ be $\mathbf{q'}=\{q_1, q_2, \cdots, q_{\frac{p}{2}+1}, 2\}$. Then $\mathbf{q'}$ is an alternating non-coprime sequence of length $p'$. By Proposition \ref{prop_mainthm}, the resulting tessellation of $S_g$ by $p'$-gons is a quotient of the building that is diffeomorphic to $I_{p', \mathbf{q'}}$. However, it follows that $I_{p', \mathbf{q'}}$ is diffeomorphic to $_{p, \mathbf{q}}$ since ${q'}_{p'}=2$ and  $\mathbf{q}$ is 2-symmetric. Thus, there exists a lattice $\Gamma_{p, \mathbf{q}, g}$ which is corresponding to the quotient of $I_{p', \mathbf{q'}, g}$.
\end{proof}

\begin{corollary}\label{cor_exist_3}
    Let $p$ be an even integer and $g$ be an integer such that $F=\displaystyle\frac{8(g-1)}{p-4}$ is an odd composite integer. Then, if the length $p$ sequence $\mathbf{q}$ is 4-symmetric by $m$ and the two gcd's in the definition of an alternating non-coprime sequence(Definition \ref{def_altseq}) are even, a lattice $\Gamma_{p, \mathbf{q}, g}$ exists.
\end{corollary}

\begin{proof}
    From the tessellation method presented in Figure \ref{fig: tessellation when F is not divisible by 4}, the boundary loops with odd lengths are generated by the red and blue edges in Figure \ref{fig: divide 4 pieces}. 
    \begin{figure}
        \centering
        \includegraphics[width=0.8\linewidth]{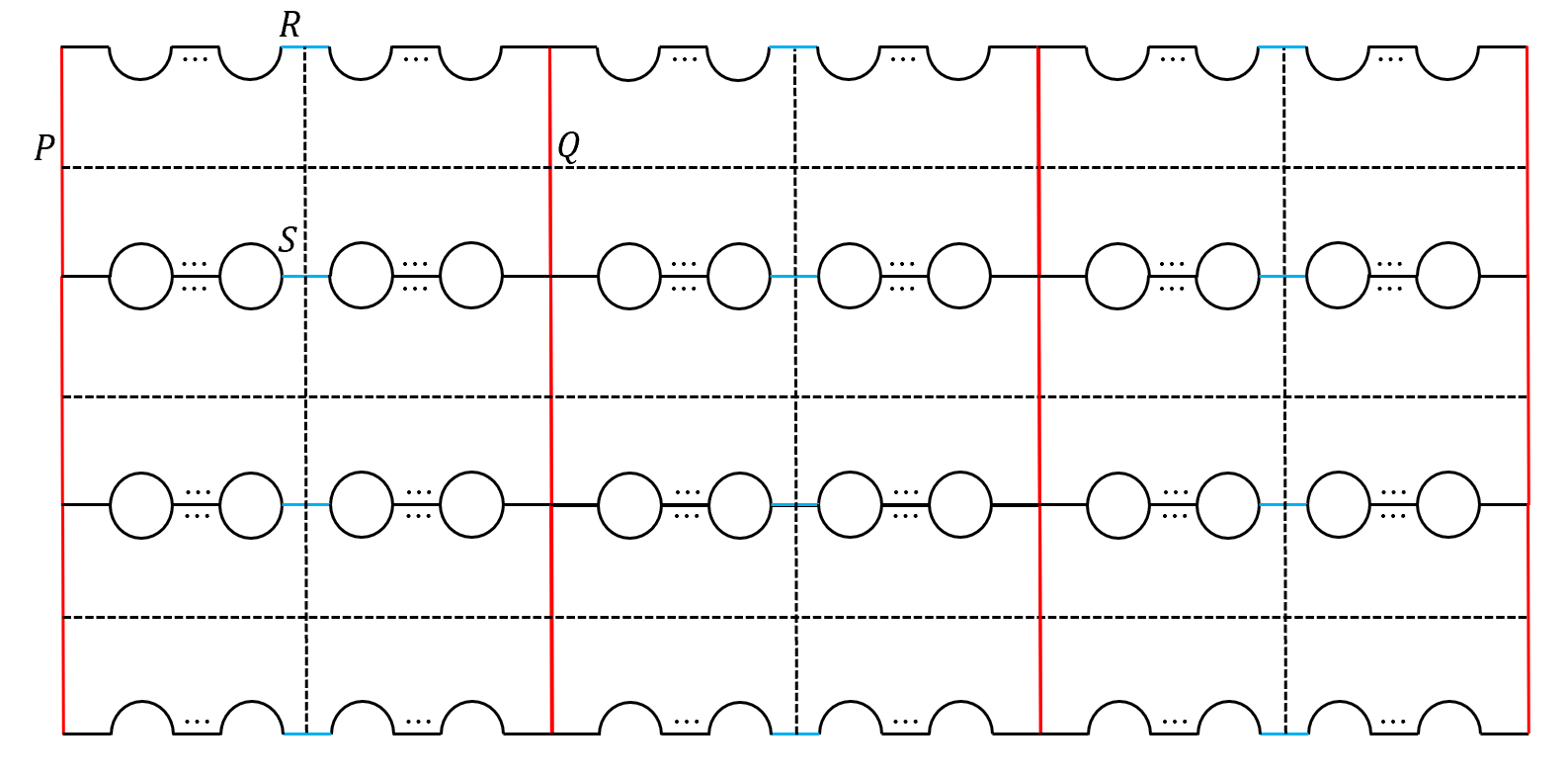}
        \caption{$F$=9 case: number of $p$-gons is $F$}
        \label{fig: divide 4 pieces}
    \end{figure}
    Similarly to the proof of Corollary \ref{cor_exist_2}, divide each $p$-gon into four pieces by the dotted lines in Figure \ref{fig: divide 4 pieces}. Then, consider the resulting tessellation as being comprised of $4F$ faces. Each face is a right-angled polygon with $p'=\frac{p}{4}+3$, and $p'$ is even since $p$ is congruent to 4 modular 8. By adding edges (for example, $PQ$ and $RS$ in Figure \ref{fig: divide 4 pieces}), every odd-length boundary geodesic loop in the original tessellation becomes an even-length loop.
    
    Because of the same reason explained in the proof of Corollary \ref{cor_exist_2}, a lattice $\Gamma_{p', \mathbf{q'}, g}$ exists for any alternating non-coprime sequence $\mathbf{q'}$ of length $p'$. Without loss of generality, let $\mathbf{q}=\{q_1, q_2, \cdots, q_p\}$ be 4-symmetric by $m=1$. By assumption, every $q_i$ is even since gcd of $q_1, q_3, \cdots, q_{p-1}$ and gcd of $q_2, q_4, \cdots, q_p$ both are even. Then, let $\mathbf{q'}=\{q_1, q_2, \cdots, q_{\frac{p}{4}+1}, 2, 2\}$. This $\mathbf{q'}$ is an alternating non-coprime sequence of length $p'$ and $I_{p', \mathbf{q'}}$ be diffeomorphic to $I_{p, \mathbf{q}}$. Thus, $\Gamma_{p', \mathbf{q'}, g}$, which we have proved the existence, is the same as $\Gamma_{p, \mathbf{q}, g}$, which is the group we aim to find.
\end{proof}

\section{Non-existence of $\Gamma_{p, \mathbf{q}, g}$}\label{sec5}
In this section, we find some sufficient conditions for the non-existence of $\Gamma_{p, \mathbf{q}, g}$. Before stating our results, we will first introduce what happens if an element of ${\text{Isom}}^{+}(\mathbb{H})$ identifies two geodesics in $\mathbb{H}$. 

The chambers in $I_{p, \mathbf{q}}$ are divided into two types according to the orientation of edge labels, which we will call \textit{counterclockwise chamber} and \textit{clockwise chamber}. In Figure \ref{fig_classification of chamber}, white faces are counterclockwise chambers and black faces are clockwise chambers.
\begin{figure}
    \centering
    \includegraphics[width=0.5\linewidth]{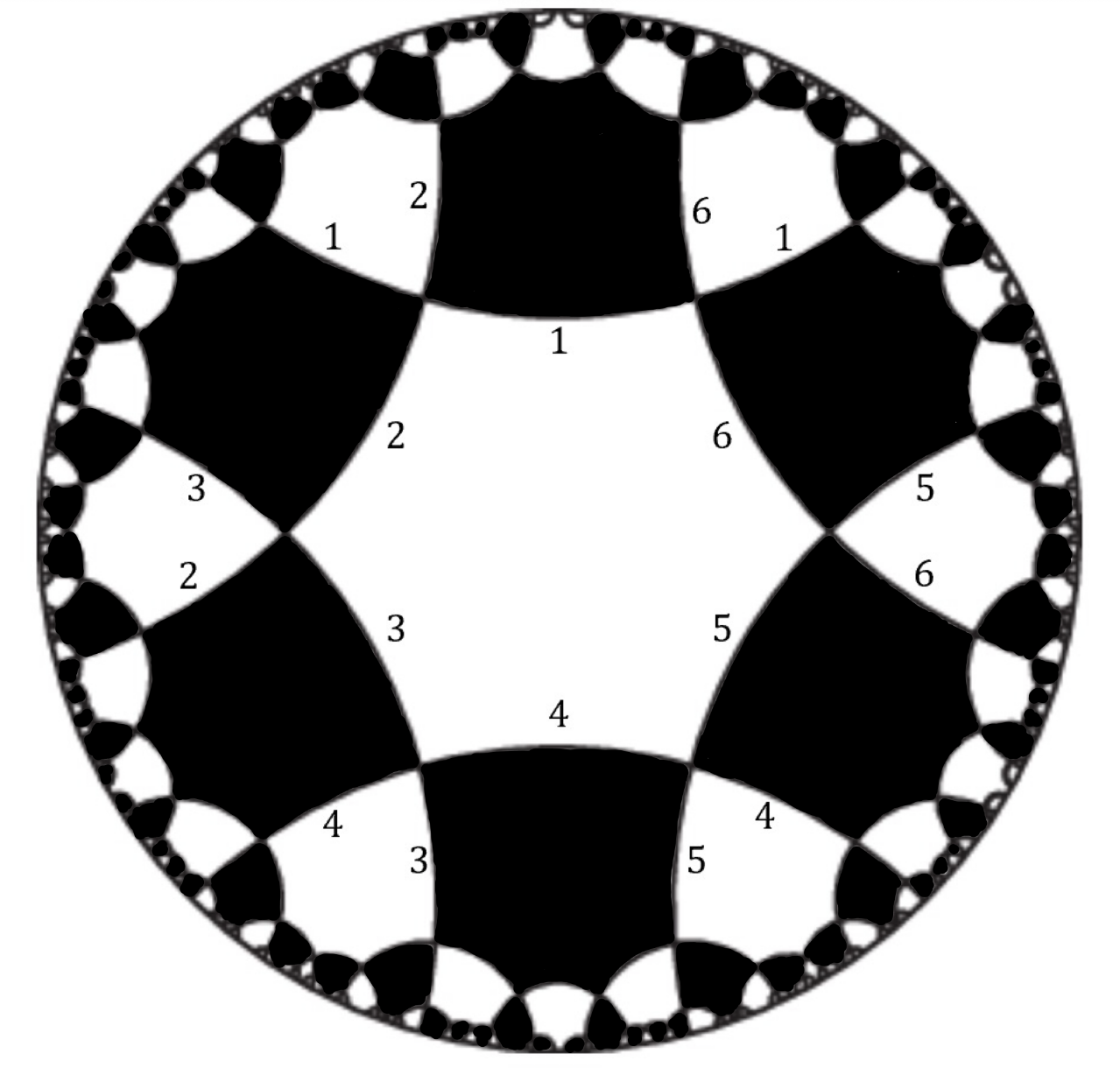}
    \caption{example of $p=6$: white faces are counterclockwise chamber, and black faces are clockwise chamber. It is easy to check that a hexagon adjacent to a counterclockwise chamber is a clockwise chamber.}
    \label{fig_classification of chamber}
\end{figure}

Define $\Gamma$ is to be the subgroup of ${\text{Isom}}^+(\mathbb{H})$ which preserves the tessellation of right-angled $p$-gons. If an isometry in $\Gamma$ sends counterclockwise[clockwise] $p$-gon to counterclockwise[clockwise, resp.] $p$-gon, define it as \textit{color preserving isometry}. Conversely, if the isometry changes the counterclockwise $p$-gons and clockwise $p$-gons, define it as \textit{color reversing isometry}. Consider such element $\gamma$ in $\Gamma$ identifying type $i$ edge to type $j$ edge. If $\gamma$ is color-preserving, $\gamma$ identifies type $i+k$ edge to type $j+k$ edge for every integer $k$. Conversely, if $\gamma$ is color reversing, it identifies the type $i+k$ edge to the type $j-k$ edge.

Denote the set of type $i$ edge in the $p$-gon tessellation of $\mathbb{H}$ as $B_i$. Since an element in $\Gamma$ sends an edge of $p$-gon tessellation to another edge, we can consider $\Gamma$ acting on $B_1 \cup B_2 \cup \cdots \cup B_p$. 

\begin{lemma}\label{lem_orbit in H}
    Let $\Gamma'$ be any subgroup in $\Gamma$. If $\Gamma' \cdot B_i$ contains at least one type $j$ edge, it contains whole $B_j$. 
\end{lemma}
\begin{proof}
    Suppose that $\gamma$, an element in $\Gamma$, sends type $i$ edge $e_0$ to type $j$ edge $e_1$. Let arbitrary another type $j$ edge as $e_2$. It follows that $\gamma^{-1}(e_2)$ is type $j$ edge from substituting $k=p$ in the above discussion, whether $\gamma$ is color preserving or reversing. Thus, $\Gamma' \cdot B_i$ contains every type $j$ edge, so it contains whole $B_j$.
\end{proof}

However, since this paper focuses on the theory of $I_{p, \mathbf{q}}$, we need to consider $\text{Aut}(I_{p, \mathbf{q}})$ instead of $\Gamma$. We can extend the above discussion to $I_{p, \mathbf{q}}$ based on the fact that any two chambers in $I_{p, \mathbf{q}}$ are contained within exactly one apartment. From now on, $B_i$ is defined as the set of type $i$ edges in $I_{p, \mathbf{q}}$, as the remaining part of this paper deals exclusively with $I_{p, \mathbf{q}}$. As discussed eariler, $\text{Aut}(I_{p, \mathbf{q}})$ can be restricted to act on $B_1 \cup B_2 \cup \cdots \cup B_p$. Consequently, the same result as in Lemma \ref{lem_orbit in H} applies to $I_{p, \mathbf{q}}$.

\begin{lemma}\label{lem_orbit in I_p,q}
    Let $\Tilde{\Gamma}$ be any subgroup of $\emph{Aut}(I_{p, \mathbf{q}})$. If $\Tilde{\Gamma}\cdot B_i$ contains at least one type $j$ edge, then it contains the entire set $B_j$.
\end{lemma}
\begin{proof}
    Let $\gamma \in \Tilde{\Gamma}$ be an element such that $\gamma(B_i)$ contains at least one type $j$ edge, and denote this edge by $e_0$. Let $e_1$ be an arbitrary edge contained in $B_j$. Since there exists an apartment containing any two chambers in $I_{p, \mathbf{q}}$, we can find an apartment $A$ that contains both $e_0$ and $e_1$. $\gamma^{-1}(A)$ is also an apartment. We can then restrict $\Tilde{\Gamma}$ to $\gamma^{-1}(A)$, which can be considered as a subgroup of $\Gamma\leq {\text{Isom}}^{+}(\mathbb{H})$, since $\Tilde{\Gamma}|_{\gamma^{-1}(A)}$ maps the apartment $\gamma^{-1}(A)$ to another apartment $A$, and both apartment are both homeomorphic to $\mathbb{H}$. By Lemma \ref{lem_orbit in H}, it follows that there exists an edge $e_2$ in $\gamma^{-1}(A)$ such that $\gamma(e_2)=e_1$.
\end{proof}

Lemma \ref{lem_odd length loop} and Lemma \ref{lem_compare thickness}, which will be stated below, play an important role in proving the necessary conditions for the non-existence of $\Gamma_{p, \mathbf{q}, g}$.

\begin{lemma}\label{lem_odd length loop}
    Let $\Tilde{\Gamma}$ be any subgroup of $\emph{Aut}(I_{p, \mathbf{q}})$. Suppose that $I_{p, \mathbf{q}} /\Tilde{\Gamma}$ is a compact hyperbolic surface containing a closed geodesic loop $l$ of odd length. Also assume that $\pi^{-1}(l)$ contains a type $i$ edge in $I_{p, \mathbf{q}}$. Then, $\Tilde{\Gamma} \cdot B_{i\pm 1}\supseteq B_{i\mp 1}$.
\end{lemma}
\begin{proof}
    Here and subsequently, let $\pi$ denote the canonical quotient map $I_{p, \mathbf{q}}\rightarrow I_{p, \mathbf{q}}/\Tilde{\Gamma}$. Let $e_0$ be a type $i$ edge contained in $\pi^{-1}(l)$. Fix an apartment $A$ that contains $e_0$, and let $l_0$ be the geodesic in $A$ that contains $e_0$. (The choice of $A$ is arbitrary among the several apartments containing $e_0$, however, $l_0$ is uniquely determined once $A$ is fixed.) The image $\pi(l_0)$ is connected and contains the edge $\pi(e_0)$ in $l$. It is easy to check $\pi(l_0)$ is same with $l$, given that $e_0$ is the unique geodesic loop in $I_{p, \mathbf{q}}/\Tilde{\Gamma}$ that contains $\pi(e_0)$. Since the length of $l$ is odd, there exists an element $\gamma \in \Tilde{\Gamma}$ that maps $e_0$ to another edge $e_1$ on $l_0$, with an even number of edges between $e_0$ and $e_1$ on $l_0$ (See Figure \ref{fig: identification}). This $\gamma$ maps the edges connected to the endpoints of $e_0$ to those connected to the endpoints of $e_1$. Since $\gamma$ fixes the geodesic $l_0$ and is orientation preserving, it maps $e_0'$ to $e_1'$ and $e_0''$ to $e_1''$ as in Figure \ref{fig: identification}. From this observation, $\gamma$ maps type $i+1$ edges to type $i-1$ edges and type $i-1$ edges to type $i+1$ edges. It follows from Lemma \ref{lem_orbit in I_p,q} that $\Tilde{\Gamma}\cdot B_{i+1}\supseteq B_{i-1}$ and $\Tilde{\Gamma}\cdot B_{i-1}\supseteq B_{i+1}$.
    \begin{figure}
        \centering
        \includegraphics[width=0.7\linewidth]{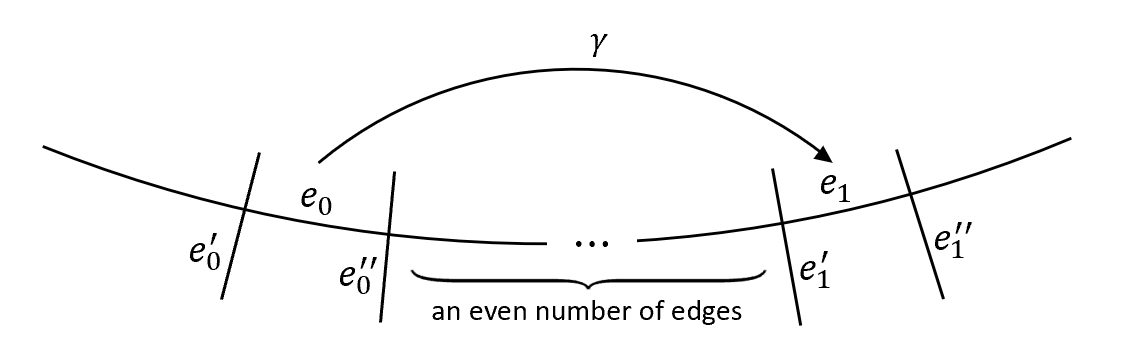}
        \caption{action by $\gamma \in \Tilde{\Gamma}$ maps $e_0$, $e_0'$, and $e_0''$ to $e_1$, $e_1'$, and $e_1''$, respectively.}
        \label{fig: identification}
    \end{figure}
\end{proof}

\begin{lemma}\label{lem_compare thickness}
    Let $\Tilde{\Gamma}$ be any subgroup of $\emph{Aut}(I_{p, \mathbf{q}})$. If $\Tilde{\Gamma}\cdot B_i \supseteq B_j$, then $q_i = q_j$ in the given sequence $\mathbf{q}$.
\end{lemma}
\begin{proof}
    Let $e_0$ be a type $i$ edge such that $\Tilde{\Gamma}\cdot e_0$ contains at least one type $j$ edge. Because of the uniqueness of local development at the midpoint of the edge $\pi (e_0)$, the thickness of $e_0$ and $e_1$ in $I_{p, \mathbf{q}}$ must be the same, which implies $q_i=q_j$.
\end{proof}

Using these Lemmas, we can prove Proposition \ref{lem: nonexistence main}, the main result of this section.

\begin{prop}\label{lem: nonexistence main}
    For given $p$, $g$, and $\mathbf{q}$ which satisfying $F=\frac{8(g-1)}{p-4}$ is integer, assume that $\Gamma_{p, \mathbf{q}, g}$ exists. If $F$ is even, however, not divided by 4, $\mathbf{q}$ is a 2-symmetric sequence. If $F$ is odd, $\mathbf{q}$ is 4-symmetric sequence.
\end{prop}
\begin{proof}
    Let $S_g$ be a hyperbolic genus $g$ surface  obtained as the quotient of $I_{p, \mathbf{q}}$ by $\Gamma_{p, \mathbf{q}}$. Let $T(S_g)$ denote the polygonal complex resulting from the $p$-gon tessellation induced by $I_{p, \mathbf{q}}$. Define $\Sigma_g$ as the complexes of groups, where a stabilizer subgroup of $\Gamma_{p, \mathbf{q}, g}$ is assigned to each vertex, edge, and face of $T(S_g)$.
    \\

    We first prove the case where $F$ is even, however, not divisible by 4. Assume that every boundary geodesic loop in $T(S_g)$ has an even length. By Proposition \ref{prop_mainthm}, a lattice $\Gamma_{p, \mathbf{q}, g}$ exists for every alternating non-coprime length $p$ sequence $\mathbf{q}$. Let $\mathbf{q}={2, 4, 6, \cdots, 2p}$, which satisfies the conditions for an alternating non-coprime sequence. Since all $q_i$ are distinct in this $\mathbf{q}$, Lemma \ref{lem_compare thickness} implies that $\Gamma_{p, \mathbf{q}, g}\cdot B_i$ cannot contain any $B_j$ for any $j\neq i$. However, if $\Gamma_{p, \mathbf{q}, g}\cdot B_i$ contains a type $j$ edge for $j\neq i$, $\Gamma_{p, \mathbf{q}, g}\cdot B_i$ must contain entire $B_j$ by Lemma \ref{lem_orbit in I_p,q}. This leads to $\Gamma_{p, \mathbf{q}, g}\cdot B_i =B_i$ for every $i$.
    
    Now, consider the canonical covering map $\pi: I_{p, \mathbf{q}}\rightarrow T(S_g)$. (This notation will be used throughout the proof.) For any boundary geodesic loop $l$ in $T(S_g)$, define the \textit{type} of $l$ as the type of the edges contained in $\pi^{-1}(l)$ in $I_{p, \mathbf{q}}$, as previously labeled. Since $\Gamma_{p, \mathbf{q}, g}\cdot B_i =B_i$, the type is well-defined for each $l$. Furthermore, it is possible to label every edge $e$ in $T(S_g)$ similarly, by assigning the type of $\pi^{-1}(e)$. 
    
    It is clear that each $p$-gon contained in $T(S_g)$ has exactly one type $i$ edge on its boundary for every $i=1, 2, \cdots, p$ by well-definedness of type. Consequently, each type has the same number of edges, meaning that there are $\frac{F}{2}$ numbers of type $i$ edges in $T(S_g)$. However, since $\frac{F}{2}$ is odd, there must exist a type $i$ boundary geodesic loop with odd length. This contradicts our assumption, and thus it suffices to prove the case where there exists a boundary geodesic loop with odd length.

    Under the above condition, it follows from Lemma \ref{lem_odd length loop} that there exists an integer $i$ such that $\Gamma_{p, \mathbf{q}, g}\cdot B_{i-1}\supseteq B_{i+1}$. Consequently, we obtain $q_{i-1}=q_{i+1}$ from Lemma \ref{lem_compare thickness}. Let $\gamma$ be an element of $\Gamma_{p, \mathbf{q}, g}$ that maps some $i-1$ type edge to $i+1$ type edge. What is left to show is to determine whether $\gamma$ is an orientation preserving or orientation reversing, $\mathbf{q}$ is 2-symmetry. If $\gamma$ is an orientation preserving, the argument at the beginning of Section \ref{sec5} shows that $\gamma$ maps a type $k$ edge to a type $k+2$ edge. Notice that since $F=\frac{8(g-1)}{p-4}$ is not divisible by 4, $p$ must be even. Thus, we deduce $q_1 = q_3 =\cdots = q_{p-1}$ and $q_2 = q_4 = \cdots = q_p$. Clearly, a sequence $\mathbf{q}$ satisfying these relations has 2-symmetry.

    We now turn to the case where $\gamma$ is orientation reversing. From the same argument used in the previous case, we conclude that $\gamma$ identifies the type $i+k$ edge with the type $i-k$ edge for every integer $k$. Following a similar approach to the former case, we get $q_{i-k}=q_{i+k}$. This implies that $\mathbf{q}$ is 2-symmetric by $i$, thereby completing the proof for the case where $F$ is even, however, not divisible by 4.
    \\

    Now, let us prove the case where $F$ is odd. Suppose that every boundary geodesic loop in $T(S_g)$ has an even length. By a similar argument to the case where $F$ is even, however, not divisible by 4, the contradiction is followed by the fact that $\frac{F}{2}$ is not an integer. Hence, there exists an boundary geodesic loop $l$ that has an odd length. Similarly to that in the proof of the former case, Lemma \ref{lem_odd length loop} and Lemma \ref{lem_compare thickness} imply that there is some $\gamma \in \Gamma_{p, \mathbf{q}, g}$ such that sends type $i-1$ edge to type $i+1$ edge for some integer $i$. If such a $\gamma$ is orientation preserving, we have already shown that $q_1=q_3=\cdots =q_{p-1}$ and $q_2=q_4=\cdots =q_p$. It is straightforward to verify that $p$ is divisible by 4 when $F$ is odd and that $\mathbf{q}$ is 4-symmetric. 
    
    Now, we only need to consider the case where every $\gamma \in \Gamma_{p, \mathbf{q}, g}$, which sends a type $i-1$ edge to a type $i+1$ edge for some $i$, is orientation reversing. Among the integers from 1 to $p$, let $i_1, i_2, \cdots, i_n$ denote all the integers for which there exists an element $\gamma$ that satisfies this condition. These values $i_1, i_2, \cdots, i_n$ can be represented as symmetry axes of the $p$-gon . Figure \ref{fig_symmetry axis} illustrates examples for $p=8$ and $p=12$. The thickness of two edges that are symmetrical about the symmetric axis is the same.
    \begin{figure}
        \centering
        \includegraphics[width=0.75\linewidth]{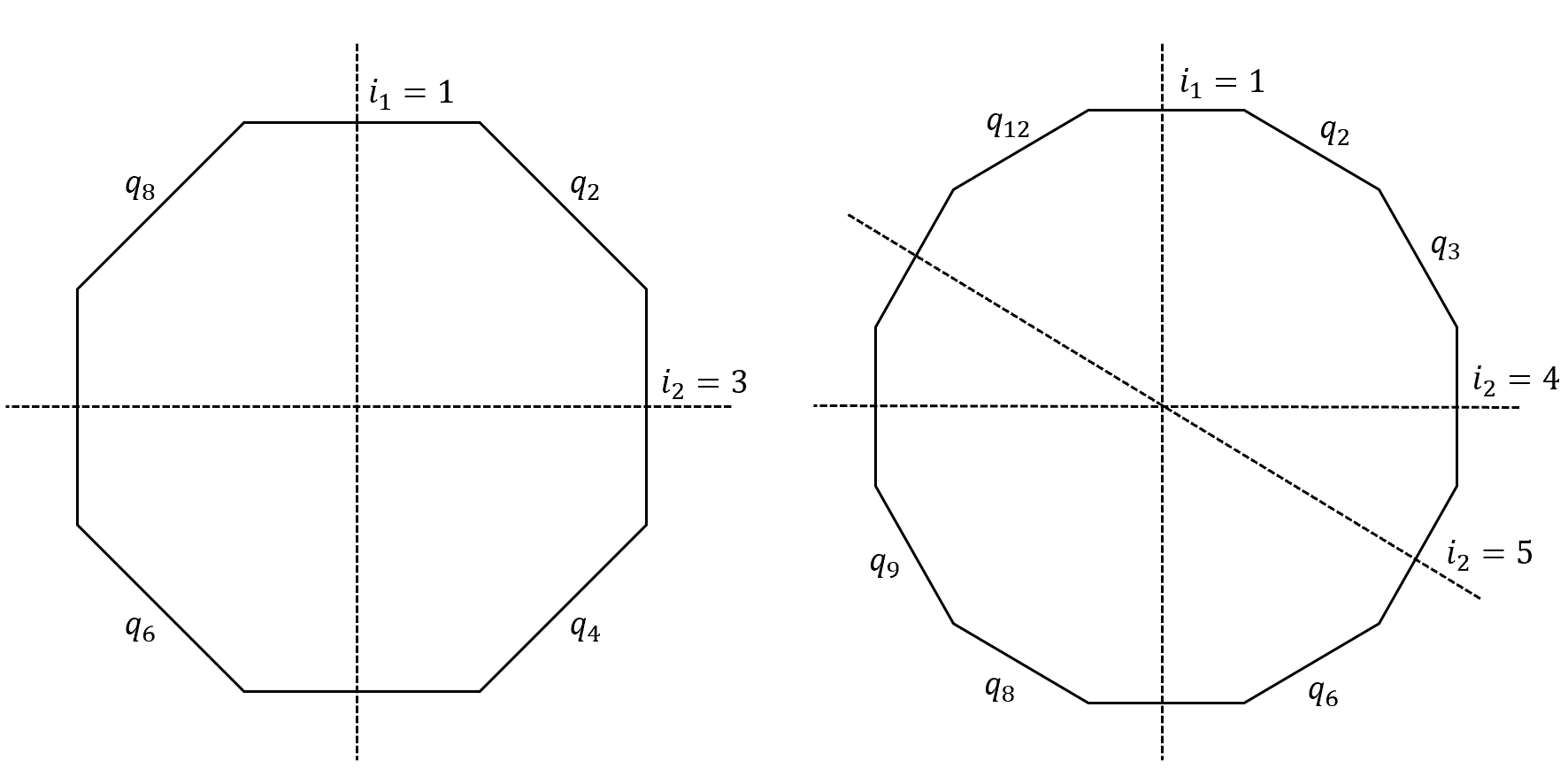}
        \caption{The left-hand side shows the example where $p=8$ and $i_1=1, i_2=3$. In this case, $q_1=q_5$, $q_3=q_7$, and $q_2=q_4=q_6=q_8$, meaning that $\mathbf{q}$ is 4-symmetric by 1. The right-hand side shows the example where $p=12$ and $i_1=1, i_2=4, i_3=5$. $q_1=q_3=\cdots=q_{11}$ and $q_2=q_4=\cdots=q_{12}$ follows from $q_{i_1-k}=q_{i_1+k}$, $q_{i_2-k}=q_{i_2+k}$, and $q_{i_3-k}=q_{i_3+k}$, as is easy to check. Hence, this case should exhibit two more symmetry axes: $i_4=2$ and $i_4=3$.}
        \label{fig_symmetry axis}
    \end{figure}

    According to \cite{martin1982classification}, if several reflection symmetries pass through a certain point on the plane, they must align with the symmetric axes of a regular polygon. Suppose that there are $n$ symmetric axes such that the angles between two adjacent axes are all the same. Remark that for these $n$ axes to fit within the $p$-gon, $p$ must be divisible by $n$. Now, assume that $n$ is odd. Denote $\frac{p}{n}$, which must be an integer by the previous remark, as $d$. Without loss of generality, let $i_1=d, i_2=2d, \cdots, i_n=nd=p$. An straightforward computation shows that $q_d=q_{2d}=\cdots=q_p =q_{d+\frac{p}{2}}=q_{2d+\frac{p}{2}}=\cdots =q_{\frac{p}{2}}$, and $\Gamma_{p, \mathbf{q}, g}\cdot B_d=B_d \cup B_{2d}\cup\cdots\cup B_{p}\cup B_{d+\frac{p}{2}}\cup B_{2d+\frac{p}{2}}\cup\cdots\cup B_{\frac{p}{2}}$. 
    
    The sum of the lengths of the boundary geodesic loop formed by the edges in $\pi(\Gamma_{p, \mathbf{q}, g}\cdot B_d)$ is $nF$. However, since $n$ and $F$ are odd, this contradicts the assumption that every boundary geodesic loop has an even length. Therefore, it remains to consider the case where $n$ is even. From the fact that there are an even number of symmetric axes with equal angles between adjacent axes, it follows directly that $\mathbf{q}$ has 4-symmetry, leading to the desired conclusion.

\end{proof}

\section*{Acknowledgement}
I would like to thank Prof. Seonhee Lim for many stimulating conversations and several helpful comments from her broad insights in hyperbolic geometry. This work was partially supported by the undergraduate research internship program at the College of Natural Sciences, Seoul National University.

\bibliographystyle{amsalpha}
\bibliography{surface_quotients_of_right-angled_hyperbolic_buildings}
\end{document}